\documentclass[11pt,a4paper,reqno,svgnames]{amsart}
  \usepackage{mathtools}
   \usepackage{scalefnt}
\usepackage[a4paper,margin=1.15in]{geometry}

\usepackage[normalem]{ulem}
\usepackage{latexsym}
\usepackage{youngtab}           
\usepackage{graphicx}             
\usepackage{color}
\usepackage{latexsym}
\usepackage{amsfonts}
\usepackage{amsxtra}
\usepackage{amsmath}
\usepackage{amssymb}
\usepackage{mathrsfs}
\usepackage{amsthm}
\usepackage{enumitem}
\usepackage{fullpage}
\usepackage{stmaryrd}
\usepackage{color}
\usepackage{booktabs}
\usepackage{lscape}
\usepackage{tikz}
   \usepackage{genyoungtabtikz}
 \usetikzlibrary{positioning,intersections,decorations}
 \usetikzlibrary{arrows,calc,through,backgrounds,matrix,decorations.pathmorphing}
\usepackage{xcolor}
\usepackage[plainpages=false, pdfpagelabels,  colorlinks=true, pdfstartview=FitV, linkcolor=DarkBlue, citecolor=DarkRed, urlcolor=blue]{hyperref}
\usepackage[color]{changebar}
\cbcolor{red}

 \theoremstyle{plain}
\newtheorem{theorem}{Theorem}[section]

\theoremstyle{definition}

\newtheorem{example}[theorem]{Example}

\numberwithin{equation}{section}

 \DeclareMathOperator{\Hom}{Hom}

\DeclareMathOperator{\Res}{Res}
\DeclareMathOperator{\Span}{Span}

\newcommand{\la}{\lambda}
\newcommand{\QQ}{{\mathbb{Q}}}

\newcommand{\FF}{\mathbb F}

\renewcommand{\ge}{\geqslant}

\renewcommand{\ge}{\geqslant}
\renewcommand{\geq}{\geqslant}
\renewcommand{\le}{\leqslant}
\renewcommand{\leq}{\leqslant}
\renewcommand{\unrhd}{\trianglerighteqslant}

\usepackage{cleveref}
 \newcommand{\Std}{{\rm Std}}

\newtheorem{thm}{Theorem}[section]
\newtheorem{cor}[thm]{Corollary}

\newtheorem{lem}[thm]{Lemma}
\newtheorem{prop}[thm]{Proposition}
\newtheorem*{prop*}{Proposition}
\newtheorem*{thm*}{Main Theorem}

\newtheorem*{cor*}{Corollary}
\newtheorem*{conj*}{Conjecture}

\newtheorem*{move1'}{Move $\mathbf{1^*}$}

\theoremstyle{remark}

\newtheorem{rmk}[thm]{Remark}

\newtheorem*{Acknowledgements*}{Acknowledgements}

\theoremstyle{definition}
\newtheorem{defn}[thm]{Definition}
\newtheorem{eg}[thm]{Example}

\crefname{defn}{Definition}{Definitions}
\crefname{thm}{Theorem}{Theorems}
\crefname{prop}{Proposition}{Propositions}
\crefname{lem}{Lemma}{Lemmas}
\crefname{cor}{Corollary}{Corollaries}
\crefname{conj}{Conjecture}{Conjectures}
\crefname{section}{Section}{Sections}
\crefname{subsection}{Subsection}{Subsections}
\crefname{eg}{Example}{Examples}
\crefname{figure}{Figure}{Figures}
\crefname{rem}{Remark}{Remarks}
\crefname{rmk}{Remark}{Remarks}
\crefname{equation}{equation}{equation}

\Crefname{defn}{Definition}{Definitions}
\Crefname{thm}{Theorem}{Theorems}
\Crefname{prop}{Proposition}{Propositions}
\Crefname{lem}{Lemma}{Lemmas}
\Crefname{cor}{Corollary}{Corollaries}
\Crefname{conj}{Conjecture}{Conjectures}
\Crefname{section}{Section}{Sections}
\Crefname{subsection}{Subsection}{Subsections}
\Crefname{eg}{Example}{Examples}
\Crefname{figure}{Figure}{Figures}
\Crefname{rem}{Remark}{Remarks}
\Crefname{rmk}{Remark}{Remarks}

\newcommand{\stt}{\mathsf t}
\newcommand{\sts}{\mathsf s}
\newcommand{\stu}{\mathsf u}
\newcommand{\stv}{\mathsf v}

\newcommand{ \Ind}{{\rm Ind}}
\renewcommand{\Res}{{\rm Res}}
\newcommand{\suchthat}{\;\ifnum\currentgrouptype=16 \middle\fi|\;} 
\title{Skew cell modules for diagram algebras}
 
\author[C. Bowman]{Christopher Bowman}
 \email{Chris.Bowman.2@city.ac.uk}

 \author{John Enyang}
\email{J.Eyang.1@city.ac.uk}
\address{Department of Mathematics, City University London, Northampton Square, London EC1V 0HB, United Kingdom }

\begin{document}
\begin{abstract}
We provide an explicit   construction of  skew cell modules for diagram algebras. 
 
    \end{abstract}
\maketitle

\section*{Introduction}

 The representation theory of the  symmetric group, $\mathfrak{S}_r$,  
is governed by its associated branching graph $\widehat{\mathfrak{S}}$ (also known as    Young's lattice). 
The set of vertices on the $r$th level   this graph, $\widehat{\mathfrak{S}}_r$, is given by the set of partitions of $r$; 
the edges connect partitions which differ by  a single node.  The first few levels of this graph are depicted in \cref{asdhjkdsf}.
 Given $\lambda$ a vertex on the $r$th level of this graph, 
 there is an associated \(\mathbb{Q}\mathfrak{S}_r\)-module, \(S^{\mathbb{Q}}_r(\lambda) \), called a {\sf Specht module}.  
This module   
 has an integral   basis indexed by the set of paths in the branching graph which 
 begin at  $\varnothing$ and terminate at $\lambda$; we refer to such paths as {\sf standard  tableaux} of shape \(\lambda\).  
 The Specht modules $\{S_r^{\mathbb{Q}}(\lambda) \mid \lambda \in \widehat{\mathfrak{S}}_r\}$ provide a complete set of simple \(\mathbb{Q}\mathfrak{S}_r\)-modules.

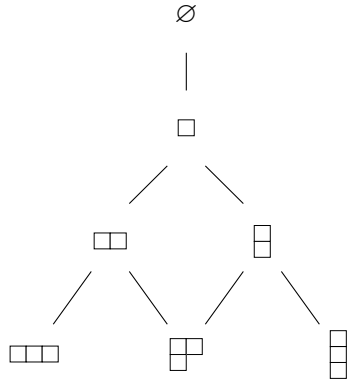
\begin{figure}[ht!]
$$\begin{tikzpicture}[scale=0.5]
          \begin{scope}  
     \draw (0,3) node {   $\varnothing $  };   
              \draw (0,0) node { $  \Yboxdim{6pt}\gyoung(;) $  };   
    \draw (-2,-3) node   {  $ \Yboxdim{6pt}\gyoung(;;)	$}		; 
       \draw (+2,-3) node   {   $ \Yboxdim{6pt}\gyoung(;,;)$	}		;
       \draw (0,-6) node    {     \Yboxdim{6pt}\gyoung(;;,;)}		;
      \draw (-4,-6) node    {    \Yboxdim{6pt}\gyoung(;;;)}		;    
        \draw (4,-6) node    {    \Yboxdim{6pt}\gyoung(;,;,;)}		;
     \draw (0.0,1) -- (0,2);  
                          \path (-1.5,-2) edge[decorate]  node[auto] { \ } (-0.5,-1);   
                                   \path (0.5,-1) edge[decorate]  node[auto] { \ } (1.5,-2);  
                          \path (-0.5,-8.2+3)  edge[decorate]  node[auto] { \ } (-1.5,-6.8+3) ;
                         \path (1.5,-6.8+3) edge[decorate]  node[auto] { \ } (0.5,-8.2+3);
  \path (-3.5,-8.2+3)  edge[decorate]  node[auto] { \ } (-2.5,-6.8+3) ;  
     \path (2.5,-6.8+3) edge[decorate]  node[auto] { \ } (3.5,-8.2+3);   
             \end{scope}\end{tikzpicture}
$$
\caption{The first few  levels of  $\widehat{\mathfrak{S}}$.  }
\label{asdhjkdsf}
\end{figure}

  For \(1\leq s \leq r\), we consider \(\mathfrak{S}_s \) as a subgroup of \(\mathfrak{S}_r\) under the embedding \(\mathfrak{S}_s \subseteq \mathfrak{S}_{r-s} \times \mathfrak{S}_s \subseteq \mathfrak{S}_r\). 
   In other words, $\mathfrak{S}_s$ is the subgroup of permutations of  the set $\{r-s+1,   \ldots ,r-1, r\}\subseteq \{1,   \ldots , r\}$.  
Generalising the above, for any two vertices   \(\lambda \in\widehat{\mathfrak{S}}_{r-s}\) and \(\nu \in\widehat{\mathfrak{S}}_r\), we may define the {\sf skew Specht module}
\begin{align*} 
S^\QQ_s( \nu\setminus\lambda )  := {\rm Hom}_{\mathfrak{S}_{r-s}}
(S^\QQ_{r-s}({\lambda}) ,{\rm Res}_{\mathfrak{S}_{r-s} }^{\mathfrak{S}_{r} }S^\QQ_r(\nu ) ). 
\end{align*}
 The {skew Specht module} $S^\QQ_s(\nu\setminus\lambda)$ 
has an integral  basis indexed by the set of paths in the branching graph which 
 begin at the vertex  $\lambda $ and terminate at the vertex $\nu$; we refer to such paths as 
{\sf skew   tableaux} of shape $\nu\setminus\lambda$.

 Skew Specht modules 
  carry the structure of   simple modules for affine Hecke algebras \cite{Ram} and play an important role in the graded representation theory of the symmetric groups and their affine analogues, see \cite{Muth}.    
  Most importantly for us,    one obtains a 
particularly natural interpretation of the   skew-Kostka     and   Littlewood--Richardson coefficients
 as  the multiplicities 
\begin{align*}
  K_{\lambda,\mu}^\nu  = \dim_{\mathbb{Q}}
{\rm Hom}_{\mathfrak{S}_{s}}
(
  {\rm ind}_{\mathfrak{S}_\mu}^{\mathfrak{S}_s} (\mathbb{Q})
, S^\QQ_s(\nu \setminus\lambda) )  
\qquad 
  c_{\lambda,\mu}^\nu  = \dim_{\mathbb{Q}}
{\rm Hom}_{\mathfrak{S}_{s}}
(S^\QQ_{s}(\mu), S^\QQ_s(\nu \setminus\lambda) ) 
\end{align*}
respectively (here ${\mathfrak{S}_\mu}={\mathfrak{S}_{\mu_1}}\times {\mathfrak{S}_{\mu_2}} \dots$).  We have a natural  map $\varphi: {\rm ind}_{\mathfrak{S}_\mu}^{\mathfrak{S}_s} (\mathbb{Q}) \to S^\QQ_{s}(\mu)$ and so we can immediately deduce that $c_{\lambda,\mu}^\nu \leq  K_{\lambda,\mu}^\nu$.  
In \cite{JAM},  James shows that 
${\rm Hom}_{\mathfrak{S}_{s}}
( {\rm ind}_{\mathfrak{S}_\mu}^{\mathfrak{S}_s} (\mathbb{Q}), S^\QQ_s(\nu \setminus\lambda) )  $   
 has a basis indexed by the set of  {\sf semistandard} skew tableaux of shape $\nu\setminus\lambda$ and weight $\mu$. 
James then shows  that the  homomorphisms which factor 
  through the map $\varphi$
are indexed by the subset of  these  skew   tableaux which satisfy the   {\sf lattice permutation} property.    
This provides a concrete combinatorial interpretation of the coefficients $ K_{\lambda,\mu}^\nu$ and $  c_{\lambda,\mu}^\nu$.  

%

  The   purpose of this article is to generalise the classical construction of 
skew Specht modules for the symmetric group to the setting of diagram algebras
 (such as the Brauer, walled Brauer, Temperley--Lieb, Birman--Murakami--Wenzl,  and partition algebras) over a field $\mathbb{F}$.  
The representation theory of these diagram    algebras is again controlled  by their  associated branching graphs.  Each vertex $\lambda$   of the branching  graph labels a   cell module $\Delta^{\mathbb{F}}(\lambda)$ and the paths in the branching  graph index integral bases of these cell modules.   
Given two fixed points $\lambda$ and $\nu$ in the branching graph, we provide an explicit construction of 
an associated {\sf skew cell module} $\Delta^{\mathbb{F}}(\nu\setminus \lambda)$.  We show that these skew cell modules 
possess  integral  bases  indexed by skew tableaux (paths between the two fixed vertices in the graph) exactly as in the classical case.   
  We  expect the  skew cell modules   
 to play an important role in the graded representation theory of these algebras (which has only recently begun to be explored, see \cite{Li})
 and provide a  link between finite dimensional diagrammatic algebras
  and their affine analogues.

 
%

The partition algebra over the rational numbers, $P^\QQ_s(n)$,   will provide the  motivating example of a diagram algebra throughout this paper.  
 Precisely as in the classical case, one can define the multiplicities 
  $$
 \Gamma_{\lambda,\mu}^\nu = 
\dim_{\mathbb{Q}}
{\rm Hom}_{{P}_{s}(n)}
({\rm ind}_{P^\QQ_\mu(n)}^{P^\QQ_s(n)}(\QQ)  , \Delta^\QQ_s(\nu \setminus\lambda) )
\qquad 
 \overline{g}_{\lambda,\mu}^\nu = \dim_{\mathbb{Q}}
{\rm Hom}_{{P}^\QQ_{s}(n)}
(\Delta^\QQ_{s}(\mu) , \Delta^\QQ_s(\nu \setminus\lambda) )$$
 (here $P^\QQ_{\mu}(n)=P^\QQ_{\mu_1}(n)\times P^\QQ_{\mu_2}(n)\dots$) and 
   $\overline{g}_{\lambda,\mu}^\nu \leq \Gamma_{\lambda,\mu}^\nu $ as before.   
  In this paper we prove that the coefficients $ \overline{g}_{\lambda,\mu}^\nu$ are equal to the {\sf stable Kronecker coefficients}. These coefficients  have been described as   `perhaps the most challenging, deep and mysterious objects in algebraic combinatorics' \cite{PP1}.  
On the other hand, the coefficients $ \Gamma_{\lambda,\mu}^\nu $ do not seem to have been studied anywhere in the literature.
Motivated by the classical case, we ask the following questions:
\begin{enumerate}[label={$\bullet$},leftmargin=*,itemsep=-0.05em]
\item
 can one interpret the coefficients $ \Gamma_{\lambda,\mu}^\nu $ and  $\overline{g}_{\lambda,\mu}^\nu $ in terms of the combinatorics of skew-tableaux for the partition algebra?
\item  do there exist natural generalisations of   the {\sf semistandard} and {\sf lattice permutation} conditions in this setting? 
\item    do   the coefficients  $\Gamma_{\lambda,\mu}^\nu $ provide a first step towards understanding the stable Kronecker coefficients $\overline{g}_{\lambda,\mu}^\nu $?
\end{enumerate}
We shall address these questions in an upcoming series of papers. 
 
%
%

The paper is structured as follows.  
In Section  1 we recall the  axiomatic framework    for  diagram algebras due to 
Enyang, Goodman, and  Graber.  
We focus on the construction of  Murphy bases of these diagram algebras
which are encoded via the  branching graph.
 In Section 2 we consider the restriction of a cell module down the tower of algebras; 
 we construct filtrations of these restricted modules which will be essential in Sections 3 and 4.
In Section 3 we recall the definitions of Jucys--Murphy elements and seminormal forms of  diagram algebras, following Goodman--Graber and Mathas.   We use the results of Section 2 
in order to improve upon these results (these stronger statements will then be used in Section 4). 
  Section 4 contains the main results of the paper.
 Given a cell module for a diagram algebra, we examine the action of certain Young subalgebras  on the  Murphy basis.  We hence construct integral bases for the skew modules which are indexed by skew tableaux.

\begin{Acknowledgements*}
The authors are grateful for the financial support received from the Royal Commission for the Exhibition of 1851  and  EPSRC grant EP/L01078X/1.
\end{Acknowledgements*}

\section{Diagram algebras}

For the remainder of the paper, we shall   let $R$ be an integral domain with field of fractions $\mathbb{F}$. 
In this section, we shall define  diagram algebras and recall the construction of their Murphy bases, following \cite{EG:2012}.  We first recall the definition of a cellular algebra, as in \cite{MR1376244}. 
\begin{defn}\label{c-d}
Let $R$ be an integral domain. A {\sf cellular algebra} is a tuple $(A,*,\widehat{A},\unrhd,\mathscr{A})$ where
\begin{enumerate}[label=(\arabic{*}), ref=\arabic{*},leftmargin=0pt,itemindent=1.5em]
\item $A$ is a unital $R$--algebra and $*:A\to A$ is an algebra  involution, that is, an $R$--linear anti--automorphism of $A$ such that $(x^*)^* = x$ for $x \in A$;
\item $(\widehat{A},\unrhd)$ is a finite partially ordered set, and $\Std(\lambda)$, for $\lambda\in\widehat{A}$, is a finite indexing set;
\item The set
\begin{align*}
\mathscr{A}=\big\{c_\mathsf{st}^\lambda  \ \big | \  \text{$\lambda\in\widehat{A}$ and $\mathsf{s},\mathsf{t}\in\Std(\lambda)$}\big\},
\end{align*}
is an $R$--basis for $A$, for which the following conditions hold:
\begin{enumerate}[label=(\alph{*}), ref=\alph{*},itemindent=1.5em]
\item\label{c-d-1} Given $\lambda\in\widehat{A}$, $\mathsf{t}\in\Std(\lambda)$, and $a\in A$, there exist 
coefficients $r(  a, \mathsf{t}, \mathsf{v}) \in R$, for $\mathsf{v}\in\Std(\lambda)$, such that, for all $\mathsf{s}\in\Std(\lambda)$, 
\begin{align}\label{r-act}
c_\mathsf{st}^\lambda a\equiv 
\sum_{\mathsf{v}\in\Std(\lambda)}
r( a; \mathsf{t}, \mathsf{v}) c_{\mathsf{sv}}^\lambda \mod{A^{\rhd\lambda}},
\end{align}
where $A^{\rhd\lambda}$ is the $R$--module with basis 
\begin{align*}
\big\{
c^\mu_\mathsf{st}\ \mid   \text{$\mathsf{s},\mathsf{t}\in\Std(\mu)$ and $\mu\rhd\lambda$}
\big\}.
\end{align*}
\item\label{c-d-2} If $\lambda\in\widehat{A}$ and $\mathsf{s},\mathsf{t}\in\Std(\lambda)$, then $(c_\mathsf{st}^\lambda)^*=(c_{\stt\sts}^\lambda)$.
\end{enumerate}
\end{enumerate}
The tuple $(A,*,\widehat{A},\unrhd,\mathscr{A})$ is a  {\sf cell datum} for $A$. 
   The basis $\mathscr{A}$ is called a {\sf  cellular basis} of $A$.  
\end{defn}
 
 We let $A^\FF $ denote the algebra $A\otimes_R \FF$.  We say that $A^\FF$ is a cellular algebra,
  with cell-datum $(A^\FF,*,\widehat{A},\unrhd,\mathscr{A}^\FF)$,
   where the  basis $\mathscr{A}^\FF$ is given by 
   $\{c^\lambda_{\sts\stt}	\otimes 1_\FF	\mid  \lambda\in\widehat{A},  \text{ and }\mathsf{s},\mathsf{t}\in\Std(\lambda) \}$.  When no confusion is possible,  we denote an element 
  $a \otimes 1_\FF \in A^\FF$  simply by  $a  \in A^\FF$.

 \begin{defn} \label{definition: cell module}
Let $A$ be a cellular algebra over the integral domain $R$.
Given  $\lambda\in\widehat{A}$, we define the {\sf cell module} $\Delta^R(\lambda)$  to be the right $A$--module  with $R$-basis 
  $$\{c^\lambda_{\sts\mathsf{t}} + A^{\vartriangleright \lambda} \mid \mathsf{t} \in \Std(\lambda) \},$$
  for any fixed $\sts \in\Std(\lambda)$.     
 We let  $\Delta^\FF(\lambda)$ denote the module  $\Delta^R(\lambda)\otimes_R \FF$.  
 
  \end{defn}

\begin{defn}[\mbox{\cite{MR3065998}}]
A cellular algebra, $A$,  is said to be {\sf cyclic cellular} if every cell module  is cyclic as an $A$-module.   \end{defn}

\begin{defn}
A cellular algebra $A$ over the integral domain $R$ is said to be {{\sf generically semisimple}} if 
$A^\FF=A\otimes_R \FF$ is semisimple.   
\end{defn}

\begin{rmk}Throughout this paper we shall assume that all our algebras are cyclic cellular and generically semisimple.
\end{rmk}

\subsection{Strongly coherent towers of cyclic cellular  algebras} 
In this section, we  recall the definitions and Murphy bases  of strongly coherent towers of cyclic cellular  algebras.  
 \begin{defn} \label{definition: cell-filtration}  Let $A$ be a cyclic cellular  algebra over $R$.  
 If $M$ is a right $A$--module, a  {\sf  cell-filtration}  of  $M$  is a filtration by right $A$--modules
\begin{align*}
\{ 0\}= M_0  \subseteq   M_1  \subseteq \cdots  \subseteq    M_r=M,
\end{align*}
such that $  M_{i} / M_{i-1} \cong \Delta^R(\lambda^{(i)})$ for some $\lambda^{(i)} \in \widehat{A}$.   We  say that the filtration is {\sf  order preserving} if  $\lambda^{(i)}\rhd \lambda^{(i+1)}$ in
  $\widehat{A}$  for all $i\geq 1$.  
 \end{defn}

Here and in the remainder of the paper, we will let $(A_k)_{k \ge 0}$ denote an   increasing sequences
$$
R= A_0 \subseteq A_1 \subseteq A_2 \cdots
$$
of cyclic cellular  algebras over an integral domain $R$.   
We assume that all the inclusions are unital and that the involutions are consistent; that is the involution on $A_{k+1}$, restricted to $A_k$, agrees with the involution on $A_k$.

 \begin{defn}[\cite{MR2794027,MR2774622}]  The tower of cyclic cellular  algebras $(A_k)_{k \ge 0}$ is {\sf  coherent}  if the following conditions are satisfied:
\begin{enumerate} 
\item  For each $k\ge 0$ and each cell module $\Delta^R_k(\lambda)$ for $\lambda \in \widehat{A}_k$,  the induced module $\Ind_{A_k}^{A_{k+1}}(\Delta^R_k(\la))$   has a  cell-filtration. 
\item  For each $k \ge 0$ and each  cell module $\Delta^R_{k+1}(\mu)$   for $\mu \in \widehat{A}_{k+1}$,    the restricted module $\Res_{A_k}^{A_{k+1}}(\Delta^R_{k+1}(\mu))$  has a cell-filtration. 
\end{enumerate}
The tower is called {\sf  strongly coherent} if the cell-filtrations can be chosen to be order preserving.
 \end{defn}

 \begin{defn}
Let  $(A_k)_{k \geq 0}$ denote a strongly coherent tower of cyclic cellular algebras. We define the branching graph, $\widehat{A}$, as follows.  The set of vertices on the $k$th level of $\widehat{A}$ is given by the set $\widehat{A}_k$.  Given $\lambda \in \widehat{A}_k$ and $\mu \in \widehat{A}_{k+1}$, 
there is an edge $\la \to \mu$ if and only if the module $\Delta^R_k(\la)$ appears in the cell-filtration of 
$\Res^{A_{k+1}}_{A_{k}}(\Delta^R_{k+1}(\mu))$.  
 \end{defn}

  \begin{defn}
Given   $\lambda  \in   \widehat{A}_{r-s}$, $\nu \in \widehat{A}_{r}$, we  
define a  {\sf  skew  standard   tableau} of shape $ \nu \setminus  \lambda  $ and degree $s$  to be a path $\stt$  of the form 
 $$
\lambda = \stt(0) \to \stt(1) \to    \stt(2)\to   \dots \to  \stt(s-1)\to \stt(s) = \nu, 
 $$
 in other words $\stt$ is  a path in the branching graph which begins at $\lambda$ and terminates at $\nu$.  
We let $\Std_s(\nu \setminus  \lambda)$ denote the set of all such paths; if $\lambda \in \widehat{A}_0$ then we let $\Std_r(\nu ):= \Std_r(\nu \setminus  \lambda)$.  
 \end{defn}

\begin{thm}[\cite{EG:2012}]\label{basisiss}
Let $(A_k)_{k\geq 0}$ denote a strongly coherent tower of generically semisimple cyclic cellular algebras over $R$.
 Given $r\in \mathbb{Z}_{\geq0}$, the algebra $A_r$  has a cellular basis (referred to as the {\sf Murphy basis}) indexed by paths in the branching graph as follows, 
\begin{align}\label{statementabove}
\{d_\sts u_\stt \mid \sts,\stt \in \Std_r(\nu), \nu \in \widehat{A}_r\}.
\end{align}
 Moreover, each basis element can be factorised in the following fashion,
$$ 
d_\sts=  d_{\sts(0) \to \sts(1)}d_{\sts(1) \to \sts(2)} \dots d_{\sts(r-1) \to \sts(r)}
\quad
u_\stt= u_{\stt(r-1) \to \stt(r)}u_{\stt(r-2) \to \stt(r-1)} \dots u_{\stt(0) \to \stt(1)}
$$ 
where the {\sf branching coefficients},
 $d_{\sts(k) \to \sts(k+1)}$ (respectively $u_{\stt(k) \to \stt(k+1)}$) depend only on the vertices
  $\sts(k) \in \widehat{A}_k$, $\sts(k+1) \in \widehat{A}_{k+1}$ (respectively
   $\stt(k) \in \widehat{A}_k$, $\stt(k+1) \in \widehat{A}_{k+1}$).  
\end{thm}


\begin{rmk}
By \cite[Section 3.4]{EG:2012}, we have that $u_\stt + A_r^{\vartriangleright  \nu}$ is a non-zero element of $A_r^{\unrhd \nu} / A_r^{\vartriangleright  \nu} $ 
for any $\stt \in \Std_r(\nu)$.  Therefore by condition $(3a)$ of \cref{c-d}, we have that 
$$ \Delta^R_r(\nu) \cong  \{d_\sts u_\stt + A_r^{\vartriangleright  \nu} \mid \stt \in \Std_r(\nu) \} \cong   \{u_\stt + A_r^{\vartriangleright  \nu} \mid \stt \in \Std_r(\nu) \}$$
as right $A_r$-modules,  for any fixed $\sts  \in \Std_r(\nu) $.  
We shall therefore use the simplified  the notation on the right-hand side of the above equation  for  the reminder of the  paper.  
\end{rmk}

\begin{eg}[\cite{EG:2012}]\label{ANWWERTWE}
Consider the chain of symmetric groups $(\mathfrak{S}_k)_{k\geq 0}$ under the obvious embedding.  The branching graph $\widehat{\mathfrak{S}}$ is the usual Young graph, with 
$\widehat{\mathfrak{S}}_k$ equal to the set of partitions of $k$. 
 There is an edge $\la \to \mu$ if $\mu$ is obtained from $\la$ by adding a single node.  The first few levels of this graph are depicted in \hyperref[asdhjkdsf]{Figure~\ref*{asdhjkdsf}}.
  The resulting cellular  basis  (as in \hyperref[basisiss]{Theorem~\ref*{basisiss}}) is the Murphy basis from \cite{MR1327362}.

\end{eg}

\begin{defn}[see \cite{MR2774622}]\label{orders on paths}
 We have two orderings on   $ \Std_s(\nu\setminus\la)$, as follows.  
   Let  
   $\la \in \widehat{A}_{r-s}$,    $\nu \in \widehat{A}_{r}$
   and $\mathsf{s}  , \mathsf{t}  \in \Std_s(\nu\setminus\la)$. 
\begin{itemize}
\item  We say that $\sts$  precedes $\stt$ in  the  {\sf  reverse lexicographic  order}  (denoted $\sts \preceq \stt$)  if $\sts = \stt$, or if  for the last index $k$ such that  $\sts(k)\ne \stt(k)$, we have $\sts(k) \vartriangleleft  \stt(k)$ in
$\widehat A_{k}$. 
\item 
 We say that $\sts$  precedes $\stt$ in the {\sf  dominance  order}  (denoted $\sts \trianglelefteq \stt$)  if
   $\sts(k)\trianglelefteq\stt(k)$  in
$\widehat A_{k}$ for all $r-s \leq k \leq r$. 
 \end{itemize}
Clearly the latter ordering is a coarsening of  the former.     
\end{defn}
\begin{rmk}
In previous works on diagram algebras, \cite{MR3092697,EG:2012,MR2794027,MR2774622}, the  reverse lexicographic ordering has been used 
as the standard ordering on tableaux.   
 In \hyperref[sec:mainresult]{Section~\ref*{sec:mainresult}}, we shall show that the reverse lexicographic  ordering can be replaced with the (coarser) dominance ordering.  
\end{rmk}
\subsection{The Jones basic construction}
    First recall that 
 an {\sf  essential idempotent} in an algebra $A$ over a ring $R$ is an element $e$ such that 
$e^2 = \delta e$ for some non--zero $\delta \in R$.   

 \begin{defn}[\cite{EG:2012}]Let $R$ be an integral domain with field of fractions $\FF$ and consider two towers of $R$-algebras with common multiplicative identity,
\begin{align}\label{p-r}
A_0\subseteq A_1\subseteq A_2\subseteq\cdots \qquad \text{and} \qquad
H_0\subseteq H_1\subseteq H_2\subseteq\cdots.
\end{align}
Suppose  that the two towers satisfy the following list of axioms:
\begin{enumerate}
\item \label{axiom: involution on An}  There is an algebra involution $*$  on $\cup_k A_k$ such that $(A_k)^* = A_k$, and likewise, there is 
an algebra involution $*$  on $\cup_k H_k$ such that $(H_k)^* = H_k$.
\item  \label{axiom: A0 and A1}  
$A_0 = H_0 = R$   and $A_1 = H_1$  (as algebras with involution).
\item \label{axiom:  idempotent and Hn as quotient of An}
 For $k \ge 2$,  $A_k$ contains an  essential idempotent $e_{k-1}$ such that $e_{k-1}^* = e_{k-1}$ and
\break $A_k/(A_k e_{k-1} A_k)  \cong H_k$ as algebras with involution.

\item \label{axiom: en An en} For $k \ge 1$,   $e_{k}$ commutes with $A_{k-1}$ and $e_{k} A_{k} e_{k} \subseteq  A_{k-1} e_{k}$.
\item  \label{axiom:  An en}
For $k \ge 1$,  $A_{k+1} 	e_{k} = A_{k} e_{k}$,  and the map $x \mapsto x e_{k}$ is injective from
$A_{k}$ to $A_{k} e_{k}$.
\item \label{axiom: e(n-1) in An en An} For $k \ge 2$,   $e_{k-1} \in A_{k+1} e_k A_{k+1}$.
\item \label{axiom Hn coherent}  $(H_k)_{k \ge 0}$ is a strongly coherent tower of cellular algebras.
\item  \label{axiom:  Delta J}
For $k \ge 2$,   $e_{k-1} A_k  e_{k-1}  A_k  =   e_{k-1}  A_k$.
 \item \label{axiom: Hn cyclic cellular}  Each $H_k$ is a cyclic cellular algebra. 
\item  \label{axiom: semisimplicity}
For all $k\geq 0$, the algebra  $A_k^\FF : = A_k \otimes_R \FF$   is split semisimple.  
\end{enumerate}
 In this  case we say that the tower $(A_k)_{k\geq 0}$ is a tower of {\sf diagram algebras}  obtained {\sf  by reflections} from the tower of algebras $(H_k)_{k\geq 0}$.  
 \end{defn}
 
 The main theorem from \cite{EG:2012}    in which we are interested is the following.      
 
\begin{thm}[\cite{EG:2012}] \label{theorem:  closed form determination of the branching factors}
Suppose that  $(A_k)_{k\geq 0}$ is a tower of  diagram algebras   
obtained  by reflections from the tower   $(H_k)_{k\geq 0}$.  Let $\widehat{H}$ denote
 the branching graph for $(H_k)_{k\geq 0}$ and 
 let  $$\bar{d}^{(k)}_{\lambda\to \mu}\qquad \bar{u}^{(k)}_{\lambda\to \mu}$$ denote the branching coefficients for 
  $\la \in \widehat{H}_k$, 
 $\mu \in \widehat{H}_{k+1}$ and $k\in \mathbb{Z}_{\geq0}$.  
 
 In which case,
the tower of algebras $(A_k)_{k\geq 0}$ is a strongly coherent tower of generically semisimple cyclic cellular algebras.  Moreover, the branching graph $\widehat{A}$   is {\sf obtained by reflections}   from $\widehat{H}$ in the following fashion.  
  The set, $\widehat{A}_k$, of  vertices on level $k$ is given by
  $$
 \widehat{A}_k= \{(\lambda,l) \mid 0 \leq l \leq \lfloor k/2 \rfloor \text{ and }\la \in \widehat{H}_{k-2l}\}.
  $$
Given $(\la,l)\in \widehat{A}_k$ and  $(\mu,m)\in \widehat{A}_{k+1}$, there is an edge $(\la,l) \to (\mu,m)$ only if $m\in \{l,l+1\}$; moreover 
\begin{itemize}
\item $(\la,l) \to (\mu, l)$ in $\widehat{A}$ if and only if $\la \to \mu$ in $\widehat{H}$;
\item $(\la,l) \to (\mu, l+1)$ in $\widehat{A}$ if and only if $\mu \to \la$ in $\widehat{H}$.  
\end{itemize}
The branching factors  for the tower $(A_k)_{k\ge 0}$ can be given in terms of the branching factors of $(H_k)_{k\geq 0}$ in the following fashion:
\begin{itemize}
\item $d_{(\la, l  ) \to (\mu, l  )}^ {(k+1)} = \bar{d}_{ \la \to \mu}^ {(k + 1 - 2l ) }  e_{k-1}^{(l  )}$;
\item $u_ {(\la, l  ) \to (\mu, l  )} ^{(k+1)} = \bar{u}_{ \la \to \mu}^ {(k + 1 - 2l ) }  e_{k} ^ {(l ) }$;
\item $d_ {(\la, l  ) \to (\mu, l  +1)} ^{(k+1)}  =  \bar{u}_{ \mu \to \la} ^{(k-2l ) } e_{k-1} ^{(l )} $;  
\item   $u_ {(\la, l  ) \to (\mu, l  +1)}^ {(k+1)}  = \bar{d}_{ \mu \to \la}^ {(k-2l ) } e_{k}^ {(l  +1)}$; 
\end{itemize}
for $(\lambda,l) \in \widehat{A}_k$ and $(\mu,m) \in \widehat{A}_{k+1}$ and $k \in \mathbb{Z}_{\geq 0}$.  
\end{thm}

\begin{rmk}\label{dsfakjhasdfjhksadf}
We note that the map  $\widehat{A}_k \to \cup_{0 \leq l \leq \lfloor k/2\rfloor }\widehat{H}_{k-2l}$
given by $:(\la,l) \to \la$ is a bijection.  Therefore, when no confusion is possible, we identify the  indexing labels under this bijection. 
\end{rmk}
 
 \begin{eg}
 Examples of algebras which fit into the above framework include
  (cyclotomic) Brauer, walled Brauer, partition, Temperley--Lieb, and Birman--Murakami--Wenzl algebras.  
 \end{eg}
 
  \begin{figure}[ht!]
  $$  
\scalefont{0.9}\begin{tikzpicture}[scale=0.6]
          \begin{scope}  
     \draw (-3,3) node {   $\left({\scalefont{1}\varnothing} ,0 \right)$  };   
              \draw (-3,0) node { $ \left(\;\Yvcentermath1\Yboxdim{6pt}\gyoung(;) \;,0 \right) $  };   
    \draw (-3,-3) node   {  $\left(\varnothing ,1\right)$ }		; 
    \draw (0,-3) node   {  $ \left(\; \Yvcentermath1\Yboxdim{6pt}\gyoung(;;) \;, 0\right)	$}		; 
       \draw (+3,-3) node   {  $ \left(\; \Yvcentermath1\Yboxdim{6pt}\gyoung(;,;) \;, 0\right)	$}		;  
      \draw (3,-6) node    {   $ \left(\; \Yvcentermath1\Yboxdim{6pt}\gyoung(;;,;) \;, 0\right)	$}		;
      \draw (0,-6) node    {  $ \left(\; \Yvcentermath1\Yboxdim{6pt}\gyoung(;;;) \;, 0\right)	$}		;    
        \draw (6,-6) node    {  $ \left(\; \Yvcentermath1\Yboxdim{6pt}\gyoung(;,;,;) \;, 0\right)	$}		;

        \draw (-3,-6) node    {  $ \left(\; \Yvcentermath1\Yboxdim{6pt}\gyoung(;) \;, 1\right)	$}		; 
    \draw (0.0-3,0.5) -- (0-3,2.5);   
                              \path (0,-3.6) edge[decorate]  node[auto] { \ } (0,-5.4);   
                      \path (3,-3.6) edge[decorate]  node[auto] { \ } (3,-5.4);   
            \path (0.5,-3.6) edge[decorate]  node[auto] { \ } (2.5,-5.4);
                        \path (3.5,-3.6) edge[decorate]  node[auto] { \ } (5.5,-5.4);    
     \path (-0.5,-3.6) edge[decorate]  node[auto] { \ } (-2.5,-5.4);  
          \path (2.5,-3.6) edge[decorate]  node[auto] { \ } (-2,-5.4);   
                    \path (-3,-3.6) edge[decorate]  node[auto] { \ } (-3,-5.4);   
                     \path (-2.5,-0.6) edge[decorate]  node[auto] { \ } (-0.5,-2.4);  
          \path (2.5,-2.4) edge[decorate]  node[auto] { \ } (-2,-0.6);   
                    \path (-3,-2.4) edge[decorate]  node[auto] { \ } (-3,-0.6);

             \end{scope}\end{tikzpicture}
$$
\caption{The first few  levels of the branching graph, $\widehat{{B}}$, of the Brauer algebras.  }
\label{asdhjkdsf2}
\end{figure}
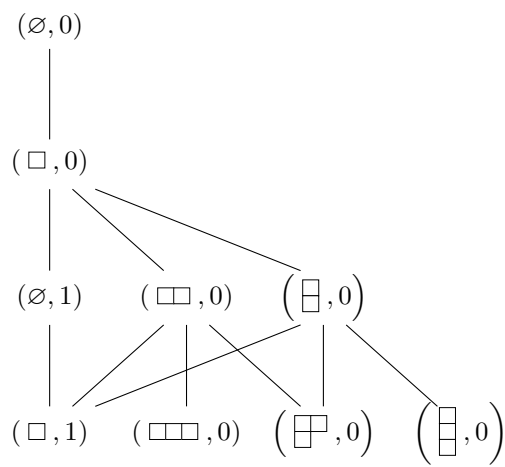

\begin{eg}[\cite{EG:2012}]
Let $B_k$ denote the Brauer algebra on $k$ strands defined in \cite{MR1503378}.
   Wenzl \cite{MR951511} showed that the   Brauer algebras form a tower of diagram algebras, $(B_k)_{k\geq 0}$, obtained `by reflections' from the  (tower of) symmetric groups (see \cref{ANWWERTWE}).
  Therefore, 
the branching graph of the Brauer algebras $(B)_{k\geq 0}$    is that obtained from $\widehat{\mathfrak{S}}$ by reflections.  This graph 
 has vertices on level $k$ given by the set 
\[
\big\{(\lambda,l) \mid 0 \leq l \leq \lfloor k/2 \rfloor \text{ and }\la \in \widehat{\mathfrak{S}}_{k-2l}\}.\}
\]
  There is an edge $(\la,l) \to (\mu,l)$  if the partition $\mu$ is obtained from  the partition  $\la$ by adding a single node.  
  There is an edge $(\la,l) \to (\mu,l+1)$  if  the partition  $\mu$ is obtained from  the partition  $\la$ by removing a single node.  
  The first few levels of  graph are depicted in \cref{asdhjkdsf2}.
 \end{eg}

\section{Filtrations of restricted modules} 
\label{section2}

For the remainder of this paper, we shall assume that $(A_k)_{k\geq0}$ is a 
strongly coherent tower of generically semisimple cyclic cellular algebras.  
  The following small technical lemma allows us to reduce the problem of constructing
 submodules (in particular, skew modules) to the semisimple case.  
\begin{lem}\label{lifting lemma}
Let $B^R$ (respectively $A^R$) denote an $R$-algebra
and let $B^\FF$  (respectively $A^\FF$) denote the algebra $B^R\otimes_R \FF$ 
(respectively  $A^R\otimes_R \FF$).  Suppose that $\iota_1:A^R \to B^R$ is an embedding of $R$-algebras.  
 Let $M^R$  denote a right $B^R$-module with $R$-basis $\lbrace m_1,\ldots,m_k\rbrace$ and let    $N^R$ denote an $R$-submodule of $M^R$ with $R$-basis
  $\lbrace m_1\ldots,m_j\rbrace\subseteq \lbrace m_1\ldots,m_k\rbrace$.
  We let $\iota_2$ denote the obvious linear inclusion of $R$-modules $\iota_2: N^R \to M^R$ and 
  let $M^\mathbb{F}=M^R\otimes_R \mathbb{F}$, $N^\mathbb{F}=N^R\otimes_R \mathbb{F}$.   
 
  Assume that 
 the subspace $N^\FF \subseteq M^\FF$ is an $A^\mathbb{F}$-submodule. Then $N^R$ carries the structure of an  $A^R$-module, with action given by
\begin{align*}
N^R\times A^R&\to N^R; \quad
(n,a) \mapsto \iota_2(n)\iota_1(a), 
\end{align*}
{for $n\in N^R,a\in A^R$.}
\end{lem}
\begin{proof}
Let $a\in A$ and $m_i\in N^R\subseteq M^R$, where $1\le i\le j$. Since $\iota_1(a)\in B^R$ and $\iota_2(m_i)\in N^R$, we have
\begin{align*}
\iota_2(m_i)\iota_1(a)=\sum_{l=1}^kr_lm_l, 
\end{align*}
 where $r_l\in R$ for $1\le l\le k$.   Hence 
\begin{align*}
\left(\iota_2(m_i)\iota_1(a)	\right)\otimes 1_\mathbb{F}=\sum_{l=1}^k(r_l\otimes 1_\mathbb{F})(m_l\otimes 1_\mathbb{F})\in M^\mathbb{F}. 
\end{align*}
Therefore $r_l=0$ whenever $j<l\le k$. Hence  $\iota_2(m_i)\iota_1(a)\in N^R$ for all $1\leq i\leq j$, as required. 
\end{proof}

 \begin{defn}
 Let $\nu\in\widehat{A}_{r}$ and $\lambda\in \widehat{A}_{r-s}$ be such that $\Std_{s}(\nu\setminus\lambda)\ne \emptyset$. We denote
\begin{align*}
M_{s,r}^R(\lambda,\nu)&=\Span_R\left\lbrace  u_\mathsf{t}+ A_r^{\rhd\nu}\suchthat\mathsf{t}\in\Std_r(\nu)\text{ and }\mathsf{t}(r-s)\unrhd \lambda\right\rbrace
\intertext{and}
U_{s,r}^R(\lambda,\nu)&=\Span_R\left\lbrace  u_\mathsf{t}+ A_r^{\rhd\nu}\suchthat\mathsf{t}\in\Std_r(\nu)\text{ and }\mathsf{t}(r-s)\rhd \lambda\right\rbrace.
\end{align*}
We set $M_{s,r}^\mathbb{F}(\lambda,\nu)= M_{s,r}^R(\lambda,\nu)\otimes_R \mathbb{F}$ 
and $U_{s,r}^\mathbb{F}(\lambda,\nu)= U_{s,r}^R(\lambda,\nu)\otimes_R \mathbb{F}$.  
 \end{defn}

The following lemma states that for a subalgebra $A^\FF_{r-s} \subseteq A^\FF_{r}$ under the 
 the tower embedding, restriction of cell modules from $A^\FF_r$ to $A^\FF_{r-s}$  
 is compatible with the cellular basis.  
Recall the conventions from \hyperref[dsfakjhasdfjhksadf]{Remark~\ref*{dsfakjhasdfjhksadf}}.

\begin{lem}\label{LEMMAONE}
Let $\nu\in\widehat{A}_r$ and $\lambda\in \widehat{A}_{r-s}$, where $\Std_{s}(\nu\setminus\lambda)\ne \emptyset$.
Then
$U_{s,r}^\mathbb{F}(\lambda,\nu)\subseteq M_{s,r}^\mathbb{F}(\lambda,\nu)\subseteq \Delta_{r}^\mathbb{F}(\nu)$ as $A_{r-s}^\mathbb{F}$-modules. Moreover 
$
M_{r,s}^\mathbb{F}(\lambda,\nu)/U_{s,r}^\mathbb{F}(\lambda,\nu)$ is isomorphic as an   $A_{r-s}^\mathbb{F}$-module to a direct sum of 
$ \sharp \Std_{s}(\nu\setminus\lambda)$ copies of   
$\Delta_{r-s}^\mathbb{F}(\lambda)$.
\end{lem}
\begin{proof}

We proceed by induction on the dominance order. Assume that $\lambda$ is maximal in $\widehat{A}_{r-s}$ subject to the condition that $\Std_{s}(\nu\setminus\lambda)\ne \emptyset$. 
For $\mu \in \widehat{A}_{r-s}$, we  let
 $E^{\mu}_{r-s}  \in {A}_{r-s}^\FF \subseteq {A}_{r}^\FF$ denote the central idempotent which projects  
 onto the simple ${A}_{r-s}^\FF$-module $\Delta^\FF_{r-s}(\mu)$.  
    Let $\mathsf{t}\in\Std_r(\nu)$, where $\mathsf{t}(r-s)=\lambda$ and suppose that $\mu\in\widehat{A}_{r-s}$, where $\mu\rhd\lambda$.  
By the cellularity of $A_{r-s}$ (and our assumption that $\mu\rhd\lambda$)   we have
\[
u_{\stt_{[0,r-s]}}E^{\mu}_{r-s}  \in A^{\vartriangleright  \lambda }_{r-s} \otimes_R \FF.
\]
Thus the maximality of $\lambda$ gives 
\begin{align}\label{dominant go to zero}
u_\stt \sum_{\mu \vartriangleright \lambda} E^{\mu}_{r-s} =
\sum_{\mu \vartriangleright \lambda} u_{\stt_{[r-s,r]}}  u_{\stt_{[0,r-s]}}E^{\mu}_{r-s} 
\in A^{\vartriangleright  \nu }_r \otimes_R \FF
\end{align}
Again  by the cellularity of $A_{r-s}$, we have that
\[
u_\mathsf{t}E^{\lambda}_{r-s}=
u_{\mathsf{t}[r-s,r]}u_{\mathsf{t}[0,r-s]}E^{\lambda}_{r-s} = u_\mathsf{t}+\sum_{\substack{\sigma\rhd\lambda\\ \mathsf{u,v}\in\Std_{r-s}(\sigma)}} r_\mathsf{u,v}u_{\mathsf{t}[r-s,r]} (d_\mathsf{u}^*u_\mathsf{v}).
\]
 Since $\sum_{\mu\rhd\lambda} E^{\mu}_{r-s}$ acts as the identity on the ideal $A_{r-s}^{\rhd \lambda} \otimes_R \FF$, to which the terms of the form $d_\mathsf{u}^*u_\mathsf{v}$ in the above sum belong, the relation $\sum_{\mu\rhd\lambda}E_{ \lambda}^{r-s}E^{\mu}_{r-s}=0$ gives
\[
0=\sum_{\mu\rhd\lambda}u_\mathsf{t}E^{\mu}_{r-s}+\sum_{\substack{\sigma\rhd\lambda\\ \mathsf{u,v}\in\Std_{r-s}(\sigma)}} r_\mathsf{u,v}u_{\mathsf{t}[r-s,r]} (d_\mathsf{u}^*u_\mathsf{v}). 
\]  
However, by equation~\eqref{dominant go to zero}, we have 
\[
\sum_{\mu\rhd\lambda}u_\mathsf{t}E^{\mu}_{r-s}\equiv 0\mod A_{r}^{\rhd \nu} \otimes_R \FF.
\]
Hence 
\(
 {u}_\mathsf{t} E^{ \lambda }_{r-s}=  {u}_\mathsf{t} \mod A_{r}^{\rhd \nu} \otimes_R \FF. 
\)
Therefore as an $A_{r-s}^\mathbb{F}$-module,  
\[
\Span_\mathbb{F}\left\lbrace  u_\mathsf{t}+ A_{r}^{\rhd\nu}\otimes_R \FF \suchthat\mathsf{t}\in\Std_r(\nu)\text{ and }\mathsf{t}(r-s)=\lambda\right\rbrace
\]
is isomorphic to a direct sum of $\sharp\Std_{r-s}(\nu\setminus\lambda)$ copies of $\Delta_{r-s}^\mathbb{F}(\lambda)$, as required. 

Next, we relax the assumption on the maximality of $\lambda$ and let $\mathsf{t}\in \Std_r(\nu)$, where $\mathsf{t}(r-s)=\lambda$. As before, we may write
\[
u_\mathsf{t}E^{\lambda}_{r-s}= u_\mathsf{t}+\sum_{\substack{\sigma\rhd\lambda\\ \mathsf{u,v}\in\Std_{r-s}(\sigma)}} r_\mathsf{u,v}u_{\mathsf{t}[r-s,r]} (d_\mathsf{u}^*u_\mathsf{v})
\]
where 
\begin{align}\label{rightmost term}
0=\sum_{\mu\rhd\lambda}u_\mathsf{t}E^{\mu}_{r-s}+\sum_{\substack{\sigma\rhd\lambda\\ \mathsf{u,v}\in\Std_{r-s}(\sigma)}} r_\mathsf{u,v}u_{\mathsf{t}[r-s,r]} (d_\mathsf{u}^*u_\mathsf{v}).
\end{align}
However, by the inductive hypothesis,
\[
\Delta_r^\mathbb{F}(\nu)\sum_{\mu\rhd\lambda} E^{\mu}_{r-s} = U_{s,r}^\mathbb{F}(\lambda,\nu)
\]
is a right $A_{r-s}^\mathbb{F}$-submodule of $\Delta_r^\mathbb{F}(\nu)$. Hence the rightmost sum in equation~\eqref{rightmost term} can  be rewritten as
\[
\sum_{\substack{\sigma\rhd\lambda\\ \mathsf{u,v}\in\Std_s(\sigma)}} r_\mathsf{u,v}u_{\mathsf{t}[r-s,r]} d_\mathsf{u}^*u_\mathsf{v}\equiv 
\sum_{\substack{\mathsf{s}\in \Std_r(\nu)\\ \mathsf{s}(r-s)\rhd \lambda}}r_\mathsf{s}u_\mathsf{s}\mod A_r^{\rhd\nu} \otimes_R \FF.
\]
Therefore,
\[
 {u}_\mathsf{t}E_{r-s}^\lambda = {u}_\mathsf{t}
  + \sum_{\{\sts \mid \sts(r-s) \vartriangleright \lambda\} } 
 r_\sts u_\sts \mod A_{r}^{\rhd \nu} \otimes_R \FF. 
\]
Therefore $M_{r,s}^\mathbb{F}(\lambda,\nu)/U_{r,s}^\mathbb{F}(\lambda,\nu)$ is isomorphic as an $A^\FF_{r-s}$-module to 
a direct sum of $\sharp\Std_s(\nu\setminus\lambda)$ copies of $\Delta_{r-s}^\mathbb{F}(\lambda)$.  The result follows by induction.  
\end{proof}

By \cref{LEMMAONE,lifting lemma}, we immediately deduce that there exists a $A_{r-s}$-module structure on $M_{s,r}^R(\lambda,\nu) / U_{s,r}^R(\lambda,\nu)$.   This is the first step towards providing a 
 cell-filtration of  $M_{s,r}^R(\lambda,\nu) / U_{s,r}^R(\lambda,\nu)$ as an $A_{r-s}$-module.  

\begin{cor}\label{corollary on restriction}
Let $\nu\in\widehat{A}_{r}$ and $\lambda\in \widehat{A}_{r-s}$ be such that $\Std_{s}(\nu\setminus\lambda)\ne \emptyset$. 
Then
$U_{s,r}^R(\lambda,\nu)\subseteq M_{s,r}^R(\lambda,\nu)\subseteq \Delta_{r}^R(\nu)$ as $A_{r-s} $-modules.   
\end{cor}

 The following technical  lemma will be useful in the proof of \cref{restriction by dominance}, below.

\begin{lem}\label{first rest}
Let $0\le s \leq  r$, $\lambda\in \widehat{A}_{r-s}$ and $\nu\in\widehat{A}_r$. Let $\stt\in\Std_r(\nu)$, where $\stt(r-s)=\lambda$. Let $a\in A_{r-s}$ and suppose that 
\begin{align}\label{adsjklsljkfadsjklfdsjklsdafkjlasdfk}
u_{\mathsf{t}}a\equiv \sum_{\mathsf{s}\in\Std_{r-s}(\lambda)} r(a;\mathsf{t}_{[0,{r-s}]},\mathsf{s})u_{\mathsf{t}_{[{r-s},r]}}u_\mathsf{s}+\sum_{\substack{\mathsf{z}\in\Std_r(\nu)\\ \mathsf{z}_{[r-s,r]}\rhd\mathsf{t}_{[{r-s},r]} }} r_\mathsf{z}u_\mathsf{z}\mod A_{r}^{\rhd\nu},
\end{align}
where 
the scalars $r(a;\mathsf{t}_{[0,{r-s}]},\mathsf{s}) \in R$ are the structure constants for   $\Delta_{r-s}^R(\lambda) $ determined  by 
\begin{align}\label{a action}
u_{\mathsf{t}_{[0,r-s]}}a=\sum_{\mathsf{s}\in\Std_{r-s}(\lambda)} r(a;\mathsf{t}_{[0,r-s]},\mathsf{s})u_\mathsf{s}+\sum_{\substack{\sigma \rhd\lambda\\ \mathsf{p,q}\in \Std_{r-s}(\sigma)}}r(a;\mathsf{t}_{[0,{r-s}]},\mathsf{p,q})d_\mathsf{p}^*u_\mathsf{q}.
\end{align}
 Then the rightmost sum in equation~\eqref{adsjklsljkfadsjklfdsjklsdafkjlasdfk} satisfies the relation
\[
  \sum_{\substack{\mathsf{z}\in\Std_r(\nu)\\ \mathsf{z}_{[{r-s},r]}\rhd\mathsf{t}_{[{r-s},r]} }} r_\mathsf{z}u_\mathsf{z} \equiv
\sum_{\substack{\sigma \rhd\lambda\\ \mathsf{p,q}\in \Std_{r-s}(\sigma)}}r(a;\mathsf{t}_{[0,{r-s}]},\mathsf{p,q})u_{\mathsf{t}_{[{r-s},r]}}d_\mathsf{p}^*u_\mathsf{q}\mod A_{r}^{\rhd\nu}.
\]
\end{lem}
\begin{proof}
Multiplying both sides of equation~\eqref{a action} by $u_{\mathsf{t}_{[{r-s},r]}}$ on the left gives:
\[
u_{\mathsf{t}}a=u_{\mathsf{t}_{[{r-s},r]}}u_{\mathsf{t}_{[0,{r-s}]}}a
=\;\;\;
\sum_{
\mathclap{\mathsf{s}\in\Std_{r-s}(\lambda)}
}\;\;
 r(a;\mathsf{t}_{[0,{r-s}]},\mathsf{s})u_{\mathsf{t}_{[{r-s},r]}}u_\mathsf{s} 
 +\;\;\;
 \sum_
 {\mathclap{\begin{subarray}c \sigma \rhd\lambda\\ \mathsf{p,q}\in \Std_{r-s}(\sigma)
 \end{subarray}}
 }
\;\;
 r(a;\mathsf{t}_{[0,{r-s}]},\mathsf{p,q})u_{\mathsf{t}_{[{r-s},r]}}d_\mathsf{p}^*u_\mathsf{q}.
\]
Since  $\stt (r-s)=  \sts(r-s)$    for $\mathsf{s}\in \Std_{r-s}(\lambda)$, we have that 
 $u_{\mathsf{t}_{[r-s,r]}}u_\mathsf{s} $    is actually the element of the cellular basis  of $A_r$ labelled  by the concatenation of the tableaux $\mathsf{t}_{[r-s,r]} $ and $ \mathsf{s}$. The result follows. 
\end{proof}

Finally, we conclude this section by demonstrating that the $A_{r-s}$-module 
 $M_{s,r}^R(\lambda,\nu) / U_{s,r}^R(\lambda,\nu)$ has a filtration in which every factor is isomorphic to $\Delta^R_{r-s}(\lambda)$.  
 This filtration is given by the dominance order on the skew tableaux of shape $\nu\setminus\lambda$.

\begin{prop}\label{restriction by dominance}
Let  $\lambda\in \widehat{A}_{r-s}$ and $\nu\in\widehat{A}_r$. Given  $\mathsf{t}\in\Std_r(\nu)$
and  $a\in A_{r-s}$, we have that  \begin{align*}
u_{\mathsf{t}}a\equiv \sum_{\mathsf{s}\in\Std_{r-s}(\lambda)} r(a;\mathsf{t}_{[0,r-s]},\mathsf{s})u_{\mathsf{t}_{[r-s,r]}}u_\mathsf{s}+
\;\;\;\;\sum_{
 \mathclap{
\begin{subarray}c 
\mathsf{z}\in\Std_r(\nu)\\ 
\mathsf{z}_{[r-s,r]} \rhd \stt_{[r-s,r]} 
\end{subarray}
 }
} \;\;\;
r_\mathsf{z}u_\mathsf{z}\mod A_{r}^{\rhd\nu}.
\end{align*}
where $r_{\mathsf z},  r(a;\mathsf{t}_{[0,r-s]},\mathsf{s})\in R$.   Moreover, the  coefficients $ r(a;\mathsf{t}_{[0,r-s]},\mathsf{s})$ are the structure constants   for $\Delta_{r-s}^R(\lambda)$  determined by \[
u_{\mathsf{t}_{[0,{r-s}]}}a\equiv\sum_{\mathsf{s}\in\Std_{r-s}(\lambda)} r(a;\mathsf{t}_{[0,{r-s}]},\mathsf{s})u_\mathsf{s} \mod A_{{r-s}}^{\rhd\lambda}.
\]
\end{prop}
\begin{proof}
The statement  holds for  $s=1$ and $r\in \mathbb{Z}_{\geq 0}$ by \cite[Proposition 2.18]{MR2774622}.   Assume by induction that the statement holds for $1 \leq s \leq r$, we shall show that the result therefore holds for $s+1$. 
%
 Assume that $a'\in A_{{r-s}-1}$ and $\stt(r-s)=\lambda$.
  By  the $s=1$ case   we have that
\begin{align}\label{third expa}
\begin{split}
u_{\mathsf{t}_{[0,{r-s}]}}a'&=
\;\;\;\;\;\; \sum_{
\mathclap{ \begin{subarray}c  \mathsf{s}\in\Std_{r-s}(\lambda)\\ \mathsf{s}({r-s}-1)=\mathsf{t}({r-s}-1)
 \end{subarray} }
 } \;\;\;
 r(a';\mathsf{t}_{[0,{r-s}]},\mathsf{s})u_\mathsf{s}+ 
\;\;\;\;\;\;\sum_{
\mathclap{\begin{subarray}c
 \mathsf{s}\in\Std_{r-s}(\lambda)\\ \mathsf{s}({r-s}-1)\rhd\mathsf{t}({r-s}-1)
 \end{subarray}}
 }  \;\;\;
r(a';\mathsf{t}_{[0,{r-s}]},\mathsf{s})u_\mathsf{s}  
\\ 
\;\;\; &\qquad\qquad+\sum_{
\mathclap{ \begin{subarray}c
\sigma\rhd\lambda\\ \mathsf{p,q}\in\Std_{r-s}(\sigma) 
\end{subarray}}
}\;\;\;
r(a';\mathsf{t}_{[0,{r-s}]},\mathsf{p,q})d_\mathsf{p}^*u_\mathsf{q}.\end{split}
\end{align}
where the first two (of the three) terms in \cref{third expa} 
 can be grouped together to be the first term in \cref{a action}; the 
 final term in \cref{third expa} is equal to the final term in \cref{a action}.  
As $a' \in A_{r-s-1   }\subset  A_{r-s}$, our inductive assumption implies that 
$u_\stt a'$ is of the form required in \cref{adsjklsljkfadsjklfdsjklsdafkjlasdfk}.  
 Therefore we can multiply both sides of equation~\eqref{third expa} by $u_{\mathsf{t}_{[r-s,r]}}$ on the left   
and apply \cref{first rest} in order to obtain 
\begin{align}\begin{split}
u_\mathsf{t}a'&\equiv
 \;\;\;\;\;\;\;\;\;\sum_{
\mathclap{\begin{subarray}c \mathsf{s}\in\Std_{r-s}(\lambda)\\ \mathsf{s}({r-s}-1)=
\mathsf{t}({r-s}-1) 
\end{subarray}}
}\;\;\;\;\;
 r(a';\mathsf{t}_{[0,{r-s}]},\mathsf{s})u_{\mathsf{t}_{[{r-s},r]}}u_\mathsf{s}
 + 
 \;\;\;\;\;\;\;\;\;
\sum_{
\mathclap{\begin{subarray}c \mathsf{s}\in\Std_{r-s}(\lambda)\\ \mathsf{s}({r-s}-1)\rhd\mathsf{t}({r-s}-1)
\end{subarray}}
} \;\;\;\;\;
r(a';\mathsf{t}_{[0,{r-s}]},\mathsf{s})u_{\mathsf{t}_{[{r-s},r]}}u_\mathsf{s}\\
&\qquad\qquad
+ \;\;\;\;\;\;  \sum_{\mathclap{
\begin{subarray}c \mathsf{z}\in\Std_{r }(\nu)\\
 \mathsf{z}_{[{r-s},r]}\rhd\mathsf{t}_{[{r-s},r]} 
\end{subarray}
}
 } \;\;\;\;\;
  r_\mathsf{z}u_\mathsf{z}\mod A_{r}^{\rhd\nu}.
\end{split}
\label{fourth expa}
\end{align}
By \hyperref[corollary on restriction]{Corollary~\ref*{corollary on restriction}} and our assumption that $a \in A_{r-s-1}$, the final sum on the right hand side of equation~\eqref{fourth expa} 
is over $\mathsf{z}\in\Std_{r-s}(\nu)$ such that $\mathsf{z}({r-s}-1)\unrhd\mathsf{t}({r-s}-1)$ and can therefore be expressed as:
\[
\sum_{\substack{\mathsf{z}\in\Std_{r-s}(\nu)\\ \mathsf{z}_{[{r-s}-1,r]}\rhd\mathsf{t}_{[{r-s}-1,r]} }} r_\mathsf{z}u_\mathsf{z}  
\]
modulo $ A_{r}^{\rhd\nu}$.  
Rewritten as above, the final sum on the right hand side of equation~\eqref{fourth expa} can be subsumed into the second  sum on the right-hand side of equation~\eqref{fourth expa} thereby giving \begin{align}
u_\mathsf{t}a'&\equiv
\sum_{\substack{\mathsf{s}\in\Std_{r-s}(\lambda)\\ \mathsf{s}({r-s}-1)=\mathsf{t}({r-s}-1)}} r(a';\mathsf{t}_{[0,{r-s}]},\mathsf{s})u_{\mathsf{t}_{[{r-s},r]}}u_\mathsf{s}+\sum_{\substack{\mathsf{z}\in\Std_{r-s}(\nu)\\ \mathsf{z}_{[{r-s}-1,r]}\rhd\mathsf{t}_{[{r-s}-1,r]} }} r'_\mathsf{z}u_\mathsf{z}\mod A_{r}^{\rhd\nu}.
\label{fifth expa}
\end{align}
Now, by  the inductive hypothesis  the scalars $r(a';\mathsf{t}_{[0,{r-s}]},\mathsf{s})$, for $\mathsf{s}\in\Std_{r-s}(\lambda)$ such that $\mathsf{s}({r-s}-1)=\mathsf{t}({r-s}-1)$, in equation~\eqref{fifth expa}
 are     the
structure constants for the cell-modules of the algebra 
 $A_{{r-s}-1}$; that is the coefficients $r(a';\mathsf{t}_{[0,{r-s}]},\mathsf{s})$ are determined by the following equation 
\[
u_{\mathsf{t}_{[0,{r-s}-1]}}a'\equiv
\sum_{\substack{\mathsf{s}\in\Std_{r-s}(\lambda)\\ \mathsf{s}({r-s}-1)=\mathsf{t}({r-s}-1)}} r(a';\mathsf{t}_{[0,{r-s}]},\mathsf{s})u_{\mathsf{s}_{[{r-s}-1,0]}}\mod A_{{r-s}-1}^{\rhd\mathsf{t}({r-s}-1)}.
\]
This completes the proof of the proposition. 
\end{proof}

\section{Jucys--Murphy elements for diagram algebras}  \label{subsection: preliminaries: JM}
\label{section: JM elements in coherent towers}

We recall the definition and first properties of families of Jucys--Murphy elements for diagram algebras. The action of Jucys--Murphy elements on cell modules for diagram algebras was first  considered systematically in  Goodman--Graber \cite{MR2794027,MR2774622}; motivated by work of   Mathas \cite{MR2414949}.   
 In this section, we use \hyperref[restriction by dominance]{Proposition~\ref*{restriction by dominance}} to  strengthen the  results of 
 \cite{MR2794027,MR2774622} by replacing the reverse lexicographic order on skew tableaux with 
 the dominance order on skew tableaux.

   \begin{defn}  
   Let   $(A_k)_{k \ge 0}$ be a strongly coherent tower of cellular algebras over an integral domain, $R$. 
 We say that a family of  elements $\{L_k:  k \ge 1\}$,  is an {\sf additive family} of {\sf Jucys--Murphy  elements} if 
    the following conditions hold.
 \begin{enumerate}
  \item    For all $k \ge 1$, $L_k \in A_k$, $L_k$ is invariant under the involution of $A_k$,   and  
  $L_k$ commutes with $A_{k-1}$.  In particular,   $L_i L_j = L_j L_i$ for all $1\leq i\leq j \leq k$.
  \item For all $k \ge 1$ and   $\la \in\widehat{A}_k$,   there exists  $d(\la) \in R$  such that    $ L_1 +  \cdots + L_k$   acts as the scalar   $d(\la)$  on the cell module $\Delta^R_k(\la)$.
  For $\lambda \in \widehat{A}_{0}$, we let $d(\lambda)=0$.
    \end{enumerate}
    \end{defn}
  
\begin{eg}
 The  group algebras of symmetric groups, Temperley--Lieb, Brauer, walled Brauer,   and partition algebras
 all possess additive families of Jucys--Murphy elements.
 \end{eg}

   \begin{defn}  
      Let   $(A_k)_{k \ge 0}$ be a strongly coherent tower of cellular algebras over an integral domain, $R$. 
 We say that a family of  elements $\{L_k:  k \ge 1\}$,  is a  {\sf multiplicative  family} of {\sf Jucys--Murphy  elements}
   if  the following conditions hold.
 \begin{enumerate}
  \item    For all $k \ge 1$,  
    $L_k$ is an invertible element of    $A_k$,
  $L_k$ is invariant under the involution $\ast$, 
   and  
  $L_k$ commutes with $A_{k-1}$.  In particular,   $L_i L_j = L_j L_i$ for all $1\leq i\leq j \leq k$.
  \item For all $k \ge 1$ and   $\la \in\widehat{A}_k$,   there exists  $d(\la) \in R$  such that    $ L_1   \cdots  L_k$   acts as the scalar   $d(\la)$  on the cell module $\Delta^R_k(\la)$.
  For $\lambda \in \widehat{A}_{0}$, we let $d(\lambda)=1$.
    \end{enumerate}
    \end{defn}
  
 \begin{eg}

  The Birman--Murakami--Wenzl algebra possess a multiplicative family of Jucys--Murphy elements.   
\end{eg}

In \cite[Proposition 3.7]{MR2774622} and \cite[Proposition  3.6]{MR2774622}, Goodman and Graber show that Jucys--Murphy elements act upper triangularly with respect to the reverse lexicographic  order on skew tableaux. The main ingredient in their proof of  is an analogue  of \hyperref[restriction by dominance]{Proposition~\ref*{restriction by dominance}}, obtained by replacing the dominance order on skew tableaux with the weaker reverse lexicographic order on skew tableaux. Therefore \hyperref[restriction by dominance]{Proposition~\ref*{restriction by dominance}} allows us to strengthen~\cite[Proposition 3.6]{MR2774622} and \cite[Proposition  3.7]{MR2774622} respectively by replacing the reverse lexicographic order on skew tableaux   with the dominance order on skew tableaux (\hyperref[orders on paths]{Definition~\ref{orders on paths}}). 
Therefore the subsequence applications of~\cite{MR2774622} (see for example,  \cite{MR3092697,MR2369064,MR2542221}) 
 can be strengthened by replacing the reverse lexicographic order of~\cite{MR3092697} (and the miscellaneous orders used in ~\cite{MR2369064,MR2542221}) with the dominance order on skew tableaux.

\begin{prop} \label{proposition: triangularity property of  additive JM elements}
Suppose that $\{L_k: k \ge 1\}$  is an additive  family of Jucys--Murphy elements  for the  tower $(A_k)_{k \ge 0}$.
\begin{enumerate}
\item
For $k \ge 1$  and $\la \in \widehat{A}_k$, let 
 $d(\la) \in R$  be  such that $L_1 + \cdots + L_k$  acts by the scalar $d(\la)$ on the cell module $\Delta^R_k(\la)$.  Then for all $k \ge 1$,  $\la \in \widehat{A}_k$,   $\stt \in \Std_k(\la)$, and $1 \le j \le k$,  we have
 \begin{equation} \label{equation: triangular action of Lj 2}
 {u}_\stt^\la       L_j  = c_\stt(j) {u}_\stt^\la + \sum_{\sts_{[j-1,k]} \, \rhd \, \stt_{[j-1,k]} } r_\sts  {u}_\sts^\la \mod A_{r}^{\rhd \la},
 \end{equation}
  for some elements $r_\sts \in R$  (depending on $j$ and $\stt$),
 with $c_\stt(j) = \displaystyle{d(\stt(j))}- {d(\stt(j-1))}$. 
 Here, the order $\rhd$ is the dominance order on skew tableaux.  
   \item For each $k \ge 1$,  $ L_1 +\cdots + L_k$ is in the center of $A_k$.  
   \end{enumerate}  
 \end{prop} 

\begin{proof}
One can replace all references to \cite[Proposition 2.18]{MR2774622} 
with references to \hyperref[restriction by dominance]{Proposition~\ref*{restriction by dominance}}  in the proof of \cite[Proposition 3.7]{MR2774622}    
\end{proof}

\begin{prop} \label{proposition: triangularity property of  multiplicative JM elements}
Suppose that $\{L_k: k \ge 1\}$  is a multiplicative family of Jucys--Murphy elements  for the  tower $(A_k)_{k \ge 0}$.
\begin{enumerate}
\item
For $k \ge 1$  and $\la \in \widehat{A}_k$, let 
 $d(\la) \in R$  be  such that $L_1   \cdots  L_k$  acts by the scalar $d(\la)$ on the cell module $\Delta^R_k(\la)$.  Then for all $k \ge 1$,  $\la \in \widehat{A}_k$,   $\stt \in \Std_k(\la)$, and $1 \le j \le k$,  we have
 \begin{equation} \label{equation: triangular action of Lj 2}
 {u}_\stt^\la       L_j  = c_\stt(j) {u}_\stt^\la + \sum_{\sts_{[j-1,k]}  \, \rhd \, \stt_{[j-1,k]} } r_\sts  {u}_\sts^\la \mod A_{r}^{\rhd \la},
 \end{equation}
  for some elements $r_\sts \in R$  (depending on $j$ and $\stt$),
 with $c_\stt(j) = \frac{d(\stt(j))}{d(\stt(j-1))}$. 
  Here, the order $\rhd$ is the dominance order on skew tableaux.  
   \item For each $k \ge 1$,  $ L_1  \cdots  L_k$ is in the center of $A_k$.  
   \end{enumerate}  
 \end{prop} 
 
\begin{proof}
One can replace all references to \cite[Proposition 2.18]{MR2774622} 
with references to \hyperref[restriction by dominance]{Proposition~\ref*{restriction by dominance}}  in the proof of \cite[Proposition 3.6]{MR2774622}.    
\end{proof}

\begin{defn}[\mbox{\cite{MR2414949}}]
Suppose  that the map $\cup_{\la\in \widehat{A}_k}\Std_k(\la)\to R^k$ given by $:\stt \mapsto (c_\stt(j))_{1 \le j \le k}$ is injective for all $k\geq 1$.  In this case, the Jucys--Murphy elements are said to satisfy the {\sf separation condition}.  
\end{defn}

 \begin{defn}\label{basis:1}
Let $(A_k)_{k\geq 0}$     be a strongly coherent tower of cellular algebras over $R$ and let $(L_k)_{k\geq1}$ be a set of additive or multiplicative  Jucys--Murphy elements 
satisfying the separation condition.    
Given  $\la \in\widehat{A}_k$ and   $ \stt\in\Std_k(\lambda)$,  
we define an 
element  of the algebra $ A^\FF_k $   as follows 
\begin{align*}
F_\stt =
\prod_{
1\le i\le k
}
\prod_{
\substack{
\stu \in    \Std_k(\rho) 		\\
c_\stu (i)\ne c_\stt (i)}}
\frac{L_i-c_\stu (i)}{c_\stt (i)-c_\stu (i)},
\end{align*}
where the product is taken over all $(\rho,k)\in\widehat{A}_k$.  We let   $f_\stt =u_\stt F_\stt $  and 
$f_{\sts \stt }=F_\sts  d_{\sts}u_{\stt} F_\stt $ for any pair $\sts,\stt\in \Std_k(\la)$.
 \end{defn}

\begin{prop}[\mbox{See~\cite[Section 3]{MR2414949},~\cite[Section~3]{MR2794027}}] \label{seminormal form}
Let $(A_k)_{k\geq 0}$    be a strongly coherent tower of cyclic cellular algebras over $R$ equipped with a set of additive or multiplicative   Jucys--Murphy elements satisfying the separation condition.     
\begin{enumerate}[label=(\arabic{*}), ref=\arabic{*},leftmargin=0pt,itemindent=1.5em]
\item\label{seminormal form.1}  For $k\geq 1$, the set of paths from the zeroth to the $k$th level of the branching graph indexes a full set   of pairwise orthogonal idempotents  in $A_k^\FF$. In other words 
 $$F_\sts F_\stt  =\delta_{\sts\stt}F_\sts \quad
 \sum_{\begin{subarray}c\sts\in\Std_k(\la)\\ \la\in\widehat{A}_k \end{subarray}} F_\sts=1_{{A}^\FF_k},$$
for $\sts \in \Std_k(\la)$ and $\stt \in \Std_k(\mu)$ and $\la,\mu \in \widehat{A}_k$.  For $\lambda \in \widehat{A}_k$, we have idempotents,
$$
1_{ k}^{\vartriangleright \lambda}
 = \sum_{\begin{subarray}c\sts\in\Std_k(\mu)\\ \mu \vartriangleright \la \end{subarray}} F_\sts
 \qquad
 1_{ k}^{\unrhd \lambda}
 = \sum_{\begin{subarray}c\sts\in\Std_k(\mu)\\ \mu \unrhd \la \end{subarray}} F_\sts $$
which act as the identity on the ideals 
$A_k^{\vartriangleright \lambda}\otimes_R \FF$
and 
$A_k^{\unrhd \lambda}\otimes_R \FF$
respectively.  For $\lambda \in \widehat{A}_k$, the idempotent
 $E_k^\lambda$ which acts by projection  onto the simple $A_{k}^\FF$-module $\Delta_k^\FF(\lambda)$ 
 is equal to 
 $$
E_{ k}^{  \lambda}
 = \sum_{ \sts\in\Std_k(\lambda) } F_\sts. 
 $$
\item\label{s-n-d:1} If $\stt \in\Std_k(\lambda)$ for $\la\in\widehat{A}$, then we have that
 \begin{align*}
f_{ \stt} =   u_\stt 
+\sum_{\begin{subarray}c
\stu \in\Std_k(\lambda) \\
{\stu \rhd\stt }
\end{subarray}
} r_\stu u_\stt,
\end{align*}
for scalars  $r_\stu \in \mathbb{F}$.
 \item\label{s-n-d:2}  
${\lbrace}f_{ \stt} \mid \stt \in\Std_k(\lambda){\rbrace}$ is an $\mathbb{F}$-basis for $\Delta^\mathbb{F}(\lambda)$. 
\item\label{s-n-d:4} $f_{ \stt} L_i=c_\stt (i)f_{ \stt} $ for all $ \stt \in\Std_k(\lambda)$ and $i=1,\ldots,k$. 
\item\label{s-n-d:5}  
$f_{\sts } F_{\stt  }= \delta_{\sts \stt }f_{\sts } $ for all $\sts ,\stt  \in\Std_k(\lambda)$.
\item\label{s-n-d:6}  $\langle f_\sts ,f_\stt \rangle=\delta_{\sts \stt}  \langle f_\sts ,f_\sts \rangle$ for all $\sts ,\stt \in\Std_k(\lambda)$.
\end{enumerate}
Therefore the basis in (3)  is a {\sf seminormal basis} in the sense of~\cite{MR2774622} and is unique up to a choice of scaling factors in $\mathbb{F}$. 
 
\end{prop}

\section{Skew cell modules for diagram algebras}\label{sec:mainresult}

For the remainder of this paper, we shall assume that $(A_k)_{k\geq0}$ is a 
tower of diagram algebras  equipped with a family of Jucys--Murphy elements.  
In this section, we construct   skew cell modules for diagram algebras 
and provide integral bases of these modules indexed by skew tableaux.  
    Given $\la \in \widehat{A}_{r-s}$ and  $\nu \in \widehat{A}_{r}$, we let 
$A^{\vartriangleright \nu \setminus \lambda}_{s,r}$ denote the $R$-subspace of $A_r$ spanned by 
\begin{equation}\label{asfdhjkdhlasdlhfsdlhaadf}
 \big\lbrace  d_{\sts}{u}_{\stt}  
\suchthat \sts,\stt\in \Std_r(\mu), 
 \mu  \vartriangleright \nu
\big\rbrace
\cup \big\lbrace  {u}_{\stt}  
\suchthat \stt\in \Std_r(\nu), 
 \stt(r-s)  \vartriangleright \lambda 
 \big\rbrace, 
\end{equation} 
and we let $A_{s,r}=A_{r-s}\times A_{s}\subseteq A_r$.  We extend this notation  to 
$A_{s,r}^\FF=A_{s,r} \otimes_R \FF$.

\begin{lem}\label{two inclusions}
    If  $\lambda\in \widehat{A}_{r-s}$ and $\nu\in \widehat{A}_r$, then:
\begin{enumerate}[label=(\arabic{*}), ref=\arabic{*},leftmargin=0pt,itemindent=1.5em]
\item\label{first inclusion} 
$M^\mathbb{F}_{s,r}(\lambda,\nu)\subseteq
 \Res^{A^\FF_r}_{A^\FF_{s,r}} (\Delta^\FF_r(\nu))$
  is an inclusion of $A_{s,r}^\mathbb{F}$-modules. 
\item\label{second inclusion} 
$M^R_{s,r}(\lambda,\nu)
\subseteq  \Res^{A_r}_{A_{s,r}} (\Delta^R_r(\nu))$
  is an inclusion of $A_{s,r}^R$-modules.
\end{enumerate}
\end{lem}

\begin{proof}
We first consider \eqref{first inclusion}.   We note that we have already proven the inclusion $(1)$ on the level of $A_{r-s}^\FF$-modules; it remains to check the inclusion holds on the level of $A_{s,r}^\FF$-modules.  
We claim that 
\begin{align}\label{claim for skew modules}
M^\mathbb{F}_{s,r}(\lambda,\nu)= 
\Delta_r^\mathbb{F}(\nu) 1^{\unrhd \lambda}_{ {r-s}}.
\end{align}
  Note that  the idempotent $ 1^{\unrhd \lambda}_{ {r-s}}$  is central in $A^\FF_{s,r}$ and therefore  the right-hand side of the equation carries the structure of  an $A_{s,r}^\FF$-module.
   Hence  to prove point~\eqref{first inclusion}, it is enough to show that  the relation~\eqref{claim for skew modules} holds.     
Given  $\lambda\in \widehat{A}_{r-s}$ and $\nu\in \widehat{A}_r$, 
we define 
\begin{align*}
 \overline{M}^{R}_{s,r}(\lambda,\nu)= M_{s,r}^R(\lambda,\nu)/ U_{s,r}^R(\lambda,\nu)
\qquad
\overline{M}_{s,r}^\mathbb{F}(\lambda,\nu)= \overline{M}_{s,r}^R(\lambda,\nu)\otimes_R\mathbb{F}.
\end{align*}
 By~\hyperref[corollary on restriction]{Corollary~\ref*{corollary on restriction}}, the space  $\overline{M}_{s,r}^R(\lambda,\nu)$ carries the structure of an  $A_{r-s} $-module. Moreover, by \hyperref[restriction by dominance]{Proposition~\ref*{restriction by dominance}}, 
 any total refinement of the dominance order on skew tableaux 
gives rise  to a    filtration 
\begin{align}\label{cell-filtration}
0=N_0\subseteq N_1\subseteq \cdots\subseteq N_k=\overline{M}_{s,r}^R(\lambda,\nu)
\end{align}
of   $\overline{M}_{s,r}^R(\lambda,\nu)$ by  $A_{r-s} 
$-submodules, where 
\begin{align}\label{filtration}
k=\sharp\Std_s(\nu\setminus\lambda)\qquad\text{and}\qquad N_i/N_{i-1}\cong \Delta^R_{r-s}(\lambda)\quad\text{for $1\le i\le k$.}
\end{align}
 By~\eqref{cell-filtration} and~\eqref{filtration}, we have that  
\begin{equation}\label{quotientfirst}
\overline{M}_{s,r}^\mathbb{F}(\lambda,\nu)E^{\lambda}_{r-s}=\overline{M}_{s,r}^\mathbb{F}(\lambda,\nu) 
\end{equation}as $A^\FF_{r-s}$-modules.  
In other words, for each $\mathsf{t}\in\Std_r(\nu)$ such that $\mathsf{t}(r-s)=\lambda$, there exist  $r_\mathsf{st}\in\mathbb{F}$ whereby
\begin{align}\label{funny scalars}
 {u}_\mathsf{t}E^{\lambda}_{r-s}=\sum_{\substack{\mathsf{s}\in\Std_r(\nu)\\ \mathsf{s}(r-s)=\lambda}}r_\mathsf{st} {u}_\mathsf{s}+\sum_{\substack{\mathsf{s}\in\Std_r(\nu)\\ \mathsf{s}(r-s)\rhd \lambda}}r_\mathsf{st} {u}_\mathsf{s}\mod A_{r}^{\rhd\nu} \otimes_R \FF
\end{align}
and 
\begin{align}\label{funny basis}
\overline{M}^\mathbb{F}_{s,r}(\lambda,\nu)=\Span_\mathbb{F}\big\lbrace  {u}_\mathsf{t}E^{\lambda}_{r-s} + A^{\vartriangleright \nu \setminus \lambda}_{s,r}\otimes_R \FF \suchthat \mathsf{t}(r-s)=\lambda \big\rbrace.
\end{align}
 We now consider the module $\Delta_r^\mathbb{F}(\nu)E^{\lambda}_{r-s}$.  We note that we can identify 
 $$
E^{\lambda}_{r-s}
=
 \sum_{
\begin{subarray}c
\stt \in \Std_{r-s}(\lambda)
\end{subarray}
} 
F_\stt
=
 \sum_{
\begin{subarray}c
\mu \in \widehat{A}_r \\
\stt \in \Std_r(\mu)\\
\stt(r-s)=\lambda  
\end{subarray}
} 
F_\stt
\in A_{r}^\FF
$$
  by 
  \cref{seminormal form}.  
  Therefore, in terms of  the seminormal basis of the   cell module, 
\begin{align}\label{sadfsadfasdfsdffsadsafsfadsafdsfadssdffasfsa}
\Delta_r^\mathbb{F}(\nu)E^{\lambda}_{r-s}
&=
{\rm Span}_\FF\left\{u_\sts F_\sts \mid \sts \in \Std_r(\nu)\right\} E^{\lambda}_{r-s}
 =
{\rm Span}_\FF\left\{
 f_\sts \mid \sts \in \Std_r(\nu), \sts(r-s)=\lambda \right.
 \} 
\end{align}
Hence,
as a $A_{r-s}^\mathbb{F} $-module,    
  $\Delta_r^\mathbb{F}(\nu)E^{\lambda}_{r-s}
$ decomposes as 
$\sharp\Std_s(\nu\setminus\lambda)$ copies of $ \Delta_{r-s}^\mathbb{F}(\lambda)$. 
%
%
%
%
We observe  
that \cref{funny basis} and \cref{sadfsadfasdfsdffsadsafsfadsafdsfadssdffasfsa} imply 
\begin{align}\label{sfanmhl}
\dim_\mathbb{F}\big(\Delta_r^\mathbb{F}(\nu)E^{\lambda}_{r-s}\big)=\dim_\mathbb{F}\big(\overline{M}_{s,r}^\mathbb{F}(\lambda,\nu)\big).
\end{align}
Having shown the equality of dimensions in  \cref{sfanmhl}, we may now prove the claim, \cref{claim for skew modules}, by induction on the dominance order.  
 If $\lambda$ is maximal 
 we have that $U_{s,r}^\mathbb{F}(\lambda,\nu)= 0$ and so 
 $$
\overline{M}_{s,r}^\mathbb{F}(\lambda,\nu)E^{\lambda}_{r-s}
=
 {M}_{s,r}^\mathbb{F}(\lambda,\nu)E^{\lambda}_{r-s}.
 $$
 Therefore  by \cref{quotientfirst}, we have 
$$
 {M}_{s,r}^\mathbb{F}(\lambda,\nu)E^{\lambda}_{r-s} =
 \Span_\mathbb{F}\big\lbrace  {u}_\mathsf{t}E^{\lambda}_{r-s} + A^{\vartriangleright \nu }_{ r}\otimes_R \FF \suchthat \mathsf{t}(r-s)=\lambda \big\rbrace.
$$
as an $A_{r-s}^\FF$-submodule of $\Delta_r^\FF(\nu) E_\lambda^{r-s}$.   Now, by \cref{sfanmhl}, we have 
 \[
\overline{M}_{s,r}^\mathbb{F}(\lambda,\nu)\cong M_{s,r}^\mathbb{F}(\lambda,\nu)= \Delta_r^\mathbb{F}(\nu)E^{\lambda}_{r-s},
\]
as required.  
Now by our inductive assumption (on the dominance order on ${\widehat{A}}_{r-s}$), the space  
\[
U_{s,r}^\mathbb{F}(\lambda,\nu)=\sum_{\substack{\mu\in \widehat{A}_{r-s}\\ \mu\rhd\lambda}} \Delta_r^\mathbb{F}(\nu)E^{\mu}_{r-s}
=
\Delta_r^\mathbb{F}(\nu) 1_{r-s}^{\rhd \lambda} 
\] 
  is an $A_{s,r}^\FF$-module.
 We note that $E^{\lambda}_{r-s}=1^{\unrhd \lambda}_{r-s} - 1^{\rhd \lambda}_{r-s}$ and therefore by  
 \eqref{funny scalars} and~\eqref{funny basis},   
we immediately obtain equation~\eqref{claim for skew modules}, by induction. 
This completes the proof of  \eqref{first inclusion}.
Point \eqref{second inclusion} now follows immediately by \hyperref[lifting lemma]{Lemma~\ref*{lifting lemma}}.  
\end{proof}

\begin{thm}\label{corollary for skew def}
Let  $\la \in \widehat{A}_{r-s}$,  $\nu \in \widehat{A}_{r}$, and let   $\stt^\lambda$ be any maximal element in the dominance ordering on $\Std_{r-s}(\lambda)$.   The $R$-module  
\begin{align*}
\Delta_s^R(\nu\setminus\lambda)= \Span_R\big\lbrace u_{\mathsf{u}}  {u}_{\stt^\lambda}  + A^{\vartriangleright \nu \setminus \lambda}_{s,r} \suchthat \mathsf{u}\in\Std_s(\nu\setminus\lambda)\big\rbrace 
\end{align*}
carries the structure of an $  A_{s} $-module under the identification 
$1\times A_{s} \subseteq A_{r-s} \times A_{s} \subseteq A_{r} $.  
\end{thm}
\begin{proof}
Let $\Delta_s^\mathbb{F}(\nu\setminus\lambda)=\Delta_s^R(\nu\setminus\lambda)\otimes_R\mathbb{F}$. It will suffice to show that 
\begin{align}\label{submodule over the field}
\Delta_s^\mathbb{F}(\nu\setminus\lambda)\subseteq \overline{M}_{s,r}^\mathbb{F}(\lambda,\nu)
\end{align}
as  a module for 
$1\times A_s^\mathbb{F}
 \subseteq A_{r-s}^\FF \times A_{s}^\FF 
 \subseteq A_{r}^\FF $.
 We know that $ u_{\stt^\lambda} F_{\stt^\lambda} = u_{\stt^\lambda}$ modulo $A_{r-s}^{\rhd \lambda}$ and therefore, by \cref{restriction by dominance},  we conclude that  for any $\stu \in\Std_s(\nu\setminus\lambda)$  there exist scalars 
\[
\lbrace r_{\stv\stu}\in\mathbb{F}\suchthat \mathsf{v}\in \Std_s(\nu\setminus\lambda)\text{ and } \mathsf{v} \unrhd \stu \rbrace,
\] 
which depend only on $\stv$ and $\stu$, such that 
\begin{align}\label{skews}
u_{  \mathsf{u} } u_{ {\stt^\lambda}}F_{\stt^\lambda}\equiv
u_{  \mathsf{u} } u_{ {\stt^\lambda}}+ 
\sum_{\substack{\mathsf{v}\in \Std_s(\nu\setminus\lambda)\\ \mathsf{v} \rhd \mathsf{u} }} r_\mathsf{v\stu}u_\mathsf{v}u_{\stt^\lambda}
\mod 
A^{\vartriangleright \nu \setminus \lambda}_{s,r} \otimes_R \FF.
\end{align}
  Hence the two sets
\begin{align}\label{two bases}
\{u_{  \mathsf{u} } u_{ {\stt^\lambda}}F_{\stt^\lambda} + A^{\rhd\nu\setminus\lambda}_{s,r}\otimes_R\FF \mid
\stu \in \Std_s(\nu\setminus\lambda)\}\quad \text{ and } \quad 
     \{u_{  \mathsf{u} } u_{ {\stt^\lambda}} + A^{\rhd\nu\setminus\lambda}_{s,r}\otimes_R\FF  \mid \stu \in \Std_s(\nu\setminus\lambda)\}
\end{align}
span the same $\mathbb{F}$-module. Moreover, the change of basis matrix between the two bases in~\eqref{two bases}  is uni-triangular with respect to any total refinement of the dominance ordering on standard tableaux.  By inverting this change of basis matrix  we obtain scalars 
\[
\lbrace r_\mathsf{uv}'\in\mathbb{F}\suchthat \mathsf{v}\in \Std_s(\nu\setminus \lambda)\text{ and } \mathsf{v} \unrhd \stu \rbrace,
\]
which depend only on $\stv$ and $\stu$, such that 
\begin{align}\label{inverted skews}
u_{\mathsf{u}}
u_{{\stt^\lambda} }  \equiv
u_{\mathsf{u}}u_{\stt^\lambda} F_{\stt^\lambda } + 
\sum_{\substack{\mathsf{v}\in \Std_s(\nu\setminus\lambda)\\ \mathsf{v} \rhd \mathsf{u} }} r_\mathsf{vu}'u_\mathsf{v} u_{\stt^\lambda} F_{\stt^\lambda }\mod A^{\vartriangleright \nu \setminus \lambda}_{s,r} \otimes_R \FF.
\end{align}
Therefore, $\Delta_s^\FF(\nu\setminus\lambda) = \overline{M}^\FF_{s,r}(\lambda,\nu) F_{\stt^\lambda}$
and hence \cref{submodule over the field} holds.  The result now follows by 
\cref{lifting lemma}.  
\end{proof}

  For the purposes of  book-keeping, we now collect together  
   \cref{asfdhjkdhlasdlhfsdlhaadf} and \cref{corollary for skew def} in the following definition.  
  
  \begin{defn}
Given $\la \in \widehat{A}_{r-s}$,  $\nu \in \widehat{A}_{r}$, 
we define the associated {\sf skew cell module}, $\Delta^R_s(\nu\setminus\lambda)$, for $A_s$  as follows:
\begin{align*}
\Delta_s^R(\nu\setminus\lambda)= \Span_R\big\lbrace u_{\mathsf{u}} {u}_{\stt^\lambda}  + 
A^{\vartriangleright \nu \setminus \lambda}_{s,r}\suchthat \stu \in\Std_s(\nu\setminus\lambda)\big\rbrace  
\end{align*}
where $\stt^\lambda $ is any  maximal element  in the dominance ordering on $\Std_{r-s}(\la)$ and 
$A^{\vartriangleright \nu \setminus \lambda}_{s,r}$ denotes the $R$-subspace of $A_r$ spanned by 
$ 
 \big\lbrace  d_{\sts}{u}_{\stt}  
\suchthat \sts,\stt\in \Std_r(\mu), 
 \mu  \vartriangleright \nu
\big\rbrace
\cup \big\lbrace  {u}_{\stt}  
\suchthat \stt\in \Std_r(\nu), 
 \stt(r-s)  \vartriangleright \lambda 
 \big\rbrace. 
$ 
 \end{defn}

We now consider the decomposition of skew cell modules over the field of fractions, $\FF$.
We let $\Delta^\FF_s(\nu\setminus\lambda) = \Delta^R_s(\nu\setminus\lambda) \otimes_R \FF$.

\begin{thm}\label{fdsakhjasldfhlhsdfsdfhl}
Given $\la \in \widehat{A}_{r-s}$, $\mu \in \widehat{A}_{s}$,  $\nu \in \widehat{A}_{r}$, 
 we have that
 \begin{align*}
 \Hom_{A^\FF_{r-s}\times A^\FF_s}
(\Delta_{r-s}^\FF(\lambda)\boxtimes \Delta_s^\FF(\mu), \Res^{A^\FF_{r}}_{A^\FF_{r-s}\times A^\FF_s}(\Delta_s^\FF(\nu ) ))
 \cong 
\Hom_{A^\FF_s}
(\Delta_s^\FF(\mu), \Delta_s^\FF(\nu\setminus\lambda)). 
  \end{align*}
\end{thm}
\begin{proof}
Any  $\varphi \in  \Hom_{A^\FF_{r-s}\times A^\FF_s}
(\Delta_{r-s}^\FF(\lambda)\boxtimes \Delta_s^\FF(\mu),  \Res^{A^\FF_{r}}_{A^\FF_{r-s}\times A^\FF_s}(\Delta_s^\FF(\nu ) ))$
is determined by  
$$\varphi(f_{\stt^\lambda} \times f_{\stt^\mu})=\sum_{\sts \in \Std_r(\nu)} a_\sts u_\sts \mod{A^{\rhd\nu}}_r $$
for some $a_\sts\in\FF$.   Moreover
$F_{\stt^\lambda} \cdot \varphi(f_{\stt^\lambda} \times f_{\stt^\mu})=\varphi(f_{\stt^\lambda} \times f_{\stt^\mu})$
and therefore by  \cref{inverted skews} and the remarks immediately afterward, we conclude that 
 $\varphi$ is determined by 
\begin{equation}\label{tradfkhja1}\varphi(f_{\stt^\lambda} \times f_{\stt^\mu})=\sum_{\bar{\sts} \in \Std_s(\nu\setminus\lambda)} b_{\bar{\sts}} u_{\bar{\sts}  }u_{\stt^\lambda}
 \mod{A^{\rhd\nu\setminus\lambda}}_r 
\end{equation}for some $b_{\bar{\sts}}\in\FF$.  
By definition of the module $\Delta^\FF_s(\nu\setminus\lambda)$, it is clear that any homomorphism $\psi  \in \Hom_{A^\FF_s}
(\Delta_s^\FF(\mu), \Delta_s^\FF(\nu\setminus\lambda)) $ is determined by 
\begin{equation}\label{tradfkhja2}
\psi(  f_{\stt^\mu})=\sum_{\stu \in \Std_s(\nu\setminus\lambda)} c_\stu u_\stu u_{\stt^\lambda}\mod{A^{\rhd\nu\setminus\lambda}}_r 
\end{equation}for some $c_\stu\in\FF$.  
  One can now apply the idempotent $F_{\stt^\mu}$ to both sides of \cref{tradfkhja1,tradfkhja2}
 and the result follows.  
    \end{proof}

\section{The partition algebra and the stable Kronecker coefficients }

Recall that a {\sf partition}    $\lambda $  is defined to be a weakly decreasing sequence   of non-negative integers. 
We define the {\sf degree} of the partition to be the sum   $|\lambda| = \lambda_1+\ldots +\lambda_\ell$.  
 Given
  $\lambda  $ a partition and $n\in \mathbb{N}$ sufficiently large  we  let 
  $$\lambda_{[n]}=(n-|\lambda|, \lambda_1,\lambda_2, \ldots,\lambda_{\ell}).$$  
Given 
$\lambda $ a partition of degree $r-s$, $\mu$ a partition of degree $s$ and 
$\nu$  a partition of degree less than or  equal to 
$ r $, we define the  Kronecker coefficients 
  to be the multiplicities  
 \begin{equation} 
  {g}_{\lambda_{[n]} , \mu_{[n]} } 
  ^{\nu_{[n]}}= 
  {\rm dim}_{\mathbb{Q}}{\rm Hom}_{\mathfrak{S}_n}
  ( S^{\mathbb{Q}}(\lambda_{[n]})\otimes  S^{\mathbb{Q}}(\mu_{[n]})
  ,S^{\mathbb{Q}}(\nu_{[n]})) .
\end{equation}
for  $n $ sufficiently large.
   As we
  let $n$ increase, the sequence of coefficients  obtained
\begin{enumerate}[label=(\arabic{*}), ref=\arabic{*},leftmargin=0pt,itemindent=1.5em]
\item is   weakly increasing, in other words   
  $   g^{\nu_{[n]}}_{\lambda_{[n]},\mu_{[n]}} \leq g^{\nu_{[n+1]}}_{\lambda_{[n+1]},\mu_{[n+1]}} \text{ for all $n$}$;  
  \item   and   for some  $N \in \mathbb N$ the sequence  stabilises        and  
   we obtain   $  \overline g^{\nu  }_{\lambda ,\mu } =   g^{\nu_{[N+k]}}_{\lambda_{[N+k]},\mu_{[N+k]}} $ for all $k\geq 0$. 
  \end{enumerate} 
  The  limiting values, $\overline g^{\nu  }_{\lambda ,\mu }$,  are known as the {\sf stable Kronecker coefficients}.  
 This is  illustrated in the following example. 
 \begin{example}
We have the following tensor products of Specht modules:
$$\begin{array}{cll}   
 n=2 	&\quad &    \mathbf{S}(1^2) \otimes \mathbf{S}(1^2)  = \mathbf{S}(2)  \\
 n=3  &\quad &\mathbf{S}(2,1) \otimes \mathbf{S}(2,1)  = \mathbf{S}(3) \oplus \mathbf{S}(2,1) \oplus \mathbf{S}(1^3) \\
n=4  &\quad &\mathbf{S}(3,1) \otimes \mathbf{S}(3,1)  = \mathbf{S}(4) \oplus \mathbf{S}(3,1)\oplus \mathbf{S}(2,1^2) \oplus \mathbf{S}(2^2)  
\end{array}
$$ 
 at which point the product stabilises, i.e. for all $n\geq4$, we have 
 $$\mathbf{S}(n-1,1) \otimes \mathbf{S}(n-1,1) = \mathbf{S}(n) \oplus \mathbf{S}(n-1,1)\oplus \mathbf{S}(n-2,1^2) \oplus \mathbf{S}(n-2,2).$$
   \end{example}

%
%
%
  
For $k \in \mathbb{Z}_{\geq0}$ and $n\in \mathbb{N}$, we let $P_k(n)$ denote the $\mathbb{Z}$-module with basis given by all set-partitions of $\{1,2,\dots, k,\bar1,\bar2,\dots,\bar k\}$.   A part of a set-partition is called a {\sf  block}.  
For example, 
$$d=\{\{\overline1, \overline2, \overline4, {2}, {5}\}, \{\overline3\}, \{\overline5, \overline6, \overline7, {3}, {4}, {6}, {7}\}, \{\overline8, {8}\}, \{{1}\}\},$$
is a set-partition (for $k=8$) with 5  blocks and $p(d)=3$.
  A  set-partition can be represented 
 by a  diagram  consisting of a frame with $k$ distinguished points on the northern and southern boundaries, which we call vertices.  We number the northern vertices from left to right by $\overline{1},\overline{2},\dots, \overline{k}$ and the southern vertices similarly by $1,2,\dots, k$  and connect two vertices by a path if they belong to the same block.     Note that such a diagram is not uniquely defined, two diagrams representing the set-partition $d$ above are given in Figure \ref{2diag}.

 \begin{figure}[ht]\scalefont{0.8}
\begin{tikzpicture}[scale=0.45]
  \draw (0,0) rectangle (8,3);
  \foreach \x in {0.5,1.5,...,7.5}
    {\fill (\x,3) circle (2pt);
     \fill (\x,0) circle (2pt);}
  \begin{scope}
    \draw (0.5,3) -- (1.5,0);
    \draw (7.5,3) -- (7.5,0);
    \draw (4.5,3) -- (2.5,0);
    \draw (0.5,3) arc (180:360:0.5 and 0.25);
    \draw (1.5,3) arc (180:360:1 and 0.25);
    \draw (4.5,0) arc (0:180:1.5 and 1);
    \draw (5.5,0) arc (0:180:1 and .7);
    \draw (3.5,0) arc (0:180:.5 and .25);
    \draw (6.5,0) arc (0:180:0.5 and 0.5);
    \draw (4.5,3) arc (180:360:0.5 and 0.25);
    \draw (5.5,3) arc (180:360:0.5 and 0.25);
      \draw (2.5,-0.49) node {$3$};   \draw (2.5,+3.5) node {$\overline{3}$};
                  \draw (1.5,-0.49) node {$2$};   \draw (1.5,+3.5) node {$\overline{2}$};
                           \draw (0.5,-0.49) node {$1$};   \draw (0.5,+3.5) node {$\overline{1}$};
         \draw (3.5,-0.49) node {$4$};   \draw (3.5,+3.5) node {$\overline{4}$};
                  \draw (4.5,-0.49) node {$5$};   \draw (4.5,+3.5) node {$\overline{5}$};
                           \draw (5.5,-0.49) node {$6$};   \draw (5.5,+3.5) node {$\overline{6}$};
         \draw (6.5,-0.49) node {$7$};   \draw (6.5,+3.5) node {$\overline{7}$};         \draw (7.5,-0.49) node {$8$};   \draw (7.5,+3.5) node {$\overline{8}$};
  \end{scope}
\end{tikzpicture}\quad \quad 
\begin{tikzpicture}[scale=0.45]
  \draw (0,0) rectangle (8,3);
  \foreach \x in {0.5,1.5,...,7.5}
    {\fill (\x,3) circle (2pt);
     \fill (\x,0) circle (2pt);}
  \begin{scope}
    \draw (0.5,3) -- (1.5,0);
    \draw (7.5,3) -- (7.5,0);
    \draw (5.5,3) -- (6.5,0);  \draw (1.5,0) -- (3.5,3);\draw (3.5,3) -- (4.5,0);
    \draw (0.5,3) arc (180:360:0.5 and 0.25);
    \draw (5.5,0) arc (0:180:1 and .7);
    \draw (3.5,0) arc (0:180:.5 and .25);
    \draw (6.5,0) arc (0:180:0.5 and 0.5);
    \draw (4.5,3) arc (180:360:1 and 0.7);
    \draw (5.5,3) arc (180:360:0.5 and 0.25);
         \draw (2.5,-0.49) node {$3$};   \draw (2.5,+3.5) node {$\overline{3}$};
                  \draw (1.5,-0.49) node {$2$};   \draw (1.5,+3.5) node {$\overline{2}$};
                           \draw (0.5,-0.49) node {$1$};   \draw (0.5,+3.5) node {$\overline{1}$};
         \draw (3.5,-0.49) node {$4$};   \draw (3.5,+3.5) node {$\overline{4}$};
                  \draw (4.5,-0.49) node {$5$};   \draw (4.5,+3.5) node {$\overline{5}$};
                           \draw (5.5,-0.49) node {$6$};   \draw (5.5,+3.5) node {$\overline{6}$};
         \draw (6.5,-0.49) node {$7$};   \draw (6.5,+3.5) node {$\overline{7}$};         \draw (7.5,-0.49) node {$8$};   \draw (7.5,+3.5) node {$\overline{8}$};
  \end{scope}
\end{tikzpicture}
  \caption{Two representatives of the set-partition $d$.}
\label{2diag}
\end{figure}
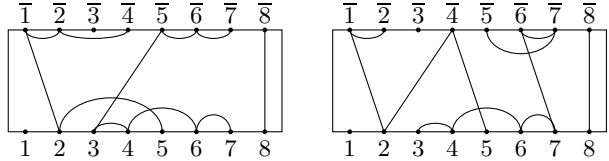

We define the product $x \cdot y$ of two diagrams $x$
and $y$ using the concatenation of $x$ above $y$, where we identify
the southern vertices of $x$ with the northern vertices of $y$.   
If there are $t$ connected components consisting only of  middle vertices, then the product is set equal to $n^t$ times the diagram  
with the middle components removed. Extending this by linearity defines a multiplication on ${P}_k(n)$.

We let $P_{k-1/2}(n)$ denote the subspace of $P_{k}(n)$ spanned by all set-partitions in which 
the integers $k$ and $\bar{k}$ appear in the same block.  This subspace is clearly closed under the multiplication and therefore is a subalgebra.  
In  \cite{EG:2012},   Enyang and Goodman  showed that the sequence of  algebras
   $$
   P_0(n) \subseteq P_{\frac{1}{2}}(n)
    \subseteq P_{1}(n)
    \subseteq P_{\frac{3}{2}}(n)
     \subseteq P_{2}(n) \subseteq \dots 
   $$
    form a tower of diagram algebras 
    obtained `by reflections' from the  (tower of) symmetric groups.   
 The Jucys--Murphy elements for the partition algebras were defined in \cite{HR}; a recursive construction of these elements is given in \cite{JUCY}.
 Therefore these algebras fit into our framework and we have the following.

\begin{prop}
    Let 
$\lambda $ be a partition of degree $r-s$, $\mu$ be  a partition of degree $s$, and 
$\nu$ be   a partition of degree less than or  equal to 
$ r $.   
 We have that
 \begin{align*}
 \overline{g}^\nu_{\lambda,\mu}= 
\Hom_{P^\QQ_s(n)}
(\Delta_s^{\mathbb{Q}}(\mu), \Delta_s^{\mathbb{Q}}(\nu\setminus\lambda)).
  \end{align*}
  for $n$ sufficiently large ($n\geq 2r$ will suffice).  
 \end{prop}
\begin{proof}
This follows immediately from \cref{fdsakhjasldfhlhsdfsdfhl} and \cite[Corollary 3.4]{BDO15}.
\end{proof}


\def\Dbar{\leavevmode\lower.6ex\hbox to 0pt{\hskip-.23ex \accent"16\hss}D}
\providecommand{\bysame}{\leavevmode\hbox to3em{\hrulefill}\thinspace}
\providecommand{\MR}{\relax\ifhmode\unskip\space\fi MR }
\providecommand{\MRhref}[2]{%
  \href{http://www.ams.org/mathscinet-getitem?mr=#1}{#2}
}
\providecommand{\href}[2]{#2}
\begin{thebibliography}{10}

\bibitem{MR1503378}
R.~Brauer, \emph{On algebras which are connected with the semisimple continuous
  groups}, Ann. of Math. (2) \textbf{38} (1937), no.~4, 857--872. \MR{1503378}

\bibitem{MR3092697}
J.~Enyang, \emph{A seminormal form for partition algebras}, J. Combin. Theory
  Ser. A \textbf{120} (2013), no.~7, 1737--1785. \MR{3092697}

\bibitem{EG:2012}
J.~Enyang and F.~M. Goodman, \emph{Cellular bases for algebras with a {J}ones
  basic construction}, Algebr. Represent. Theory (to appear).

\bibitem{MR3065998}
T.~Geetha and F.~M. Goodman, \emph{Cellularity of wreath product algebras and
  {$A$}-{B}rauer algebras}, J. Algebra \textbf{389} (2013), 151--190.
  \MR{3065998}

\bibitem{MR2794027}
F.~M. Goodman and J.~Graber, \emph{Cellularity and the {J}ones basic
  construction}, Adv. in Appl. Math. \textbf{46} (2011), no.~1-4, 312--362.
  \MR{2794027}

\bibitem{MR2774622}
\bysame, \emph{On cellular algebras with {J}ucys {M}urphy elements}, J. Algebra
  \textbf{330} (2011), 147--176. \MR{2774622}

\bibitem{MR1376244}
J.~J. Graham and G.~I. Lehrer, \emph{Cellular algebras}, Invent. Math.
  \textbf{123} (1996), no.~1, 1--34. \MR{1376244}

\bibitem{Li}
G.~Li, \emph{A {K}{L}{R} grading of the {B}rauer algebras}, arXiv:1409.1195
  (2015).

\bibitem{MR2414949}
A.~Mathas, \emph{Seminormal forms and {G}ram determinants for cellular
  algebras}, J. Reine Angew. Math. \textbf{619} (2008), 141--173, With an
  appendix by Marcos Soriano. \MR{2414949}

\bibitem{MR1327362}
G.~E. Murphy, \emph{The representations of {H}ecke algebras of type {$A\sb
  n$}}, J. Algebra \textbf{173} (1995), no.~1, 97--121. \MR{1327362}

\bibitem{Muth}
R.~Muth, \emph{Graded skew {S}pecht modules and cuspidal modules for
  {K}hovanov--{L}auda--{R}ouquier algebras of affine type ${A}$}, in
  preparation (2015).

\bibitem{Ram}
A.~Ram, \emph{Skew shape representations are irreducible}, Combinatorial and
  geometric representation theory ({S}eoul, 2001), Contemp. Math., vol. 325,
  Amer. Math. Soc., Providence, RI, 2003, pp.~161--189. \MR{1988991}

\bibitem{MR2369064}
H.~Rui and M.~Si, \emph{Discriminants of {B}rauer algebras}, Math. Z.
  \textbf{258} (2008), no.~4, 925--944. \MR{2369064}

\bibitem{MR2542221}
\bysame, \emph{Gram determinants and semisimplicity criteria for
  {B}irman-{W}enzl algebras}, J. Reine Angew. Math. \textbf{631} (2009),
  153--179. \MR{2542221}

\bibitem{MR951511}
H.~Wenzl, \emph{On the structure of {B}rauer's centralizer algebras}, Ann. of
  Math. (2) \textbf{128} (1988), no.~1, 173--193. \MR{951511}

\end{thebibliography}


\begin{thebibliography}{10}

 
  \bibitem{BDO15}  C. Bowman;  M. De Visscher;  R. Orellana,  The partition algebra and  the Kronecker coefficients, 
Trans. Amer. Math. Soc.   (2015).
 
  \bibitem{BDE}  C. Bowman;  M. De Visscher;  J. Enyang,  The co-Pieri rule for Kronecker coefficients, 
 preprint.  

\bibitem{MR1503378}
R.~Brauer, \emph{On algebras which are connected with the semisimple continuous
  groups}, Ann. of Math. (2) \textbf{38} (1937), no.~4, 857--872. 




\bibitem{JUCY}
J.~Enyang, 
Jucys--Murphy elements and a presentation for partition algebras, 
J. Algebraic Combin. {\bf 37} (2013), no. 3, 401--454. 

\bibitem{MR3092697}
J.~Enyang, \emph{A seminormal form for partition algebras}, J. Combin. Theory
  Ser. A \textbf{120} (2013), no.~7, 1737--1785. 




\bibitem{EG:2012}
J.~Enyang; F.~M. Goodman, \emph{Cellular bases for algebras with a {J}ones
  basic construction}, Algebr. Represent. Theory (to appear).

\bibitem{MR3065998}
T.~Geetha; F.~M. Goodman, \emph{Cellularity of wreath product algebras and
  {$A$}-{B}rauer algebras}, J. Algebra \textbf{389} (2013), 151--190.
  

\bibitem{MR2794027}
F.~M. Goodman; J.~Graber, \emph{Cellularity and the {J}ones basic
  construction}, Adv. in Appl. Math. \textbf{46} (2011), no.~1-4, 312--362.
 
 
\bibitem{MR2774622}
\bysame, \emph{On cellular algebras with {J}ucys {M}urphy elements}, J. Algebra
  \textbf{330} (2011), 147--176.  

\bibitem{MR1376244}
J.~J. Graham; G.~I. Lehrer, \emph{Cellular algebras}, Invent. Math.
  \textbf{123} (1996), no.~1, 1--34.  


\bibitem{HR}
T. Halverson; A. Ram, Partition algebras, European J. Combin. 26 (2005), no. 6, 869--921.

   
 \bibitem{JAM}
G. D. James,  A characteristic-free approach to the representation theory of $S_n$,
 J. Algebra  {\bf 46} (1977)  430--450. 


\bibitem{Li}
G.~Li, \emph{A {K}{L}{R} grading of the {B}rauer algebras}, arXiv:1409.1195
  (2015).

\bibitem{marbook}
P.~P. Martin, \emph{Potts models and related problems in statistical
  mechanics}, Series on Advances in Statistical Mechanics, 5, World Scientific
  Publishing Co., Inc., Teaneck, NJ, 1991.


\bibitem{MR2414949}
A.~Mathas, \emph{Seminormal forms and {G}ram determinants for cellular
  algebras}, J. Reine Angew. Math. \textbf{619} (2008), 141--173, With an
  appendix by Marcos Soriano.  

\bibitem{MR1327362}
G.~E. Murphy, \emph{The representations of {H}ecke algebras of type {$A\sb
  n$}}, J. Algebra \textbf{173} (1995), no.~1, 97--121.  

\bibitem{Muth}
R.~Muth, \emph{Graded skew {S}pecht modules and cuspidal modules for
  {K}hovanov--{L}auda--{R}ouquier algebras of affine type ${A}$}, in
  preparation (2015).
  \bibitem{PP1}
I. Pak; G. Panova, 
Bounds on Kronecker and $q$-binomial coefficients, 
arXiv1410.7087v2.

\bibitem{Ram}
A.~Ram, \emph{Skew shape representations are irreducible}, Combinatorial and
  geometric representation theory ({S}eoul, 2001), Contemp. Math., vol. 325,
  Amer. Math. Soc., Providence, RI, 2003, pp.~161--189. 

\bibitem{MR2369064}
H.~Rui; M.~Si, \emph{Discriminants of {B}rauer algebras}, Math. Z.
  \textbf{258} (2008), no.~4, 925--944.  

\bibitem{MR2542221}
\bysame, \emph{Gram determinants and semisimplicity criteria for
  {B}irman-{W}enzl algebras}, J. Reine Angew. Math. \textbf{631} (2009),
  153--179.  

\bibitem{MR951511}
H.~Wenzl, \emph{On the structure of {B}rauer's centralizer algebras}, Ann. of
  Math. (2) \textbf{128} (1988), no.~1, 173--193. 

\end{thebibliography}
\end{document}